\documentclass{amsart}
\usepackage{amssymb,amsmath, amsthm,latexsym}
\usepackage{graphicx}
\usepackage{amscd}
\usepackage{graphicx}
\usepackage{tikz-cd}
\newcommand{\cal}[1]{\mathcal{#1}}
\theoremstyle{plain}

\newtheorem{theo}{Theorem}

\newtheorem{lemma}{Lemma}[section]
\newtheorem{theorem}[lemma]{Theorem}
\newtheorem{proposition}[lemma]{Proposition}
\newtheorem{corollary}[lemma]{Corollary}
\theoremstyle{definition}
\newtheorem{definition}[lemma]{Definition}
\newtheorem{defi}[theo]{Definition}
\newtheorem{remark}[lemma]{Remark}

\parskip=\bigskipamount

\let\egthree=\phi
\let\phi=\varphi
\let\varphi=\egthree

\begin{document}
\title[Spin mapping class group]
{Generating the spin mapping class group by Dehn twists} 
\author{Ursula Hamenst\"adt}
\thanks
{AMS subject classification: 30F30, 30F60, 37B10, 37B40}
\date{May 1, 2021}

\begin{abstract}
  Let $\phi$ be a $\mathbb{Z}/2\mathbb{Z}$-spin structure
 on a closed oriented surface $\Sigma_g$ of 
 genus $g\geq 4$.
We determine a generating set of 
the stabilizer 
of $\phi$
in the mapping class group of $\Sigma_g$ 
 consisting of Dehn twists about an 
explicit collection of $2g+1$ curves on $\Sigma_g$. 
If $g=3$ then we determine a generating set 
of the stabilizer of an odd $\mathbb{Z}/4\mathbb{Z}$-spin
structure consisting of Dehn twists about a collection of 
$6$ curves. 
\end{abstract}

\maketitle

\section{Introduction}

For some $r\geq 2$,
a $\mathbb{Z}/r\mathbb{Z}$-spin structure on 
a closed surface $\Sigma_g$ of genus $g$ is a cohomology
class $\phi\in H^1(UT\Sigma_g,\mathbb{Z}/r\mathbb{Z})$ which evaluates
to one on the oriented fibre of the unit tangent bundle $UT\Sigma_g\to \Sigma_g$ 
of $\Sigma_g$. 
Such a spin structure exists for all $r$ which divide $2g-2$. If $r$ is even,
then it reduces to a $\mathbb{Z}/2\mathbb{Z}$-spin structure on $\Sigma_g$. 

A $\mathbb{Z}/2\mathbb{Z}$-spin structure 
on $\Sigma_g$ has 
a \emph{parity}, either even or odd. Thus there is a notion of parity for all
$\mathbb{Z}/r\mathbb{Z}$-spin structures with $r$ even.
If $\phi,\phi^\prime$ are two $\mathbb{Z}/r\mathbb{Z}$-spin
structures on $\Sigma_g$ so that 
either $r$ is odd or $r$ is even and the parities of
$\phi,\phi^\prime$ coincide, then
there exists an element of the mapping class
group ${\rm Mod}(\Sigma_g)$ of $\Sigma_g$ which maps $\phi$ 
to $\phi^\prime$. Hence
the stabilizers of $\phi$ and $\phi^\prime$ in ${\rm Mod}(\Sigma_g)$ 
are conjugate. 

Spin structures naturally arise in the context of abelian differentials on 
$\Sigma_g$. The moduli space of such differentials decomposes into
strata of differentials whose zeros are of the same order
and multiplicity. 
Understanding the orbifold fundamental group of such strata requires
some understanding of their projection to the mapping class group.
If the orders of the zeros of the differentials are all multiples of the same
number $r\geq 2$, then 
this quotient group preserves a $\mathbb{Z}/r\mathbb{Z}$-spin structure 
$\phi$ on $\Sigma_g$. Hence the orbifold fundamental
groups of components of strata relate to 
stabilizers ${\rm Mod}(\Sigma_g)[\phi]$ 
of spin structures $\phi$ on $\Sigma_g$.

To make such a relation explicit we define

\begin{defi}\label{curvesystem}
A \emph{curve system} on a closed surface 
$\Sigma_g$ is a finite collection of smoothly embedded
simple closed curves on $\Sigma_g$ which are 
 non-contractible and mutually not freely homotopic, 
and 
such that any two curves from this collection intersect transversely 
in at most one point.
\end{defi}

A curve system defines a \emph{curve diagram} which is a finite graph
whose vertices are the curves from the system and where two
such vertices are connected by an edge if the curves intersect.

\begin{defi}
  A curve system on $\Sigma_g$ is \emph{admissible} if it decomposes
  $\Sigma_g$ into a collection of topological disks and if its 
  curve diagram is a tree.
\end{defi}

Using a construction of Thurston and Veech (see \cite{L04} for a 
comprehensive account), admissible curve systems on $\Sigma_g$ 
give rise to abelian differentials on $\Sigma_g$, and the 
component of the stratum and hence the equivalence class of 
a spin structure (if any) it defines can be read off explicitly from 
the combinatorics of the curve system. 
This makes it desirable to investigate the subgroup of the mapping class
group generated by Dehn twists about the curves of an 
admissible curve system.

The main goal of this article is to present a systematic study of
stabilizers
of suitably chosen curves in the spin mapping class group
${\rm Mod}(\Sigma_g)[\phi]$ 
and to use this information
to build generators for this group by induction over subsurfaces.
As a main application we obtain the following.

For 
$g\geq 3$ let ${\cal C}_g$ and ${\cal V}_g$ be the 
collections of $2g+1$ nonseparating simple closed curves
on a closed surface $\Sigma_g$ of genus $g$ shown in Figure 1.
\begin{figure}[ht]
\begin{center}
\includegraphics[width=0.9\textwidth]{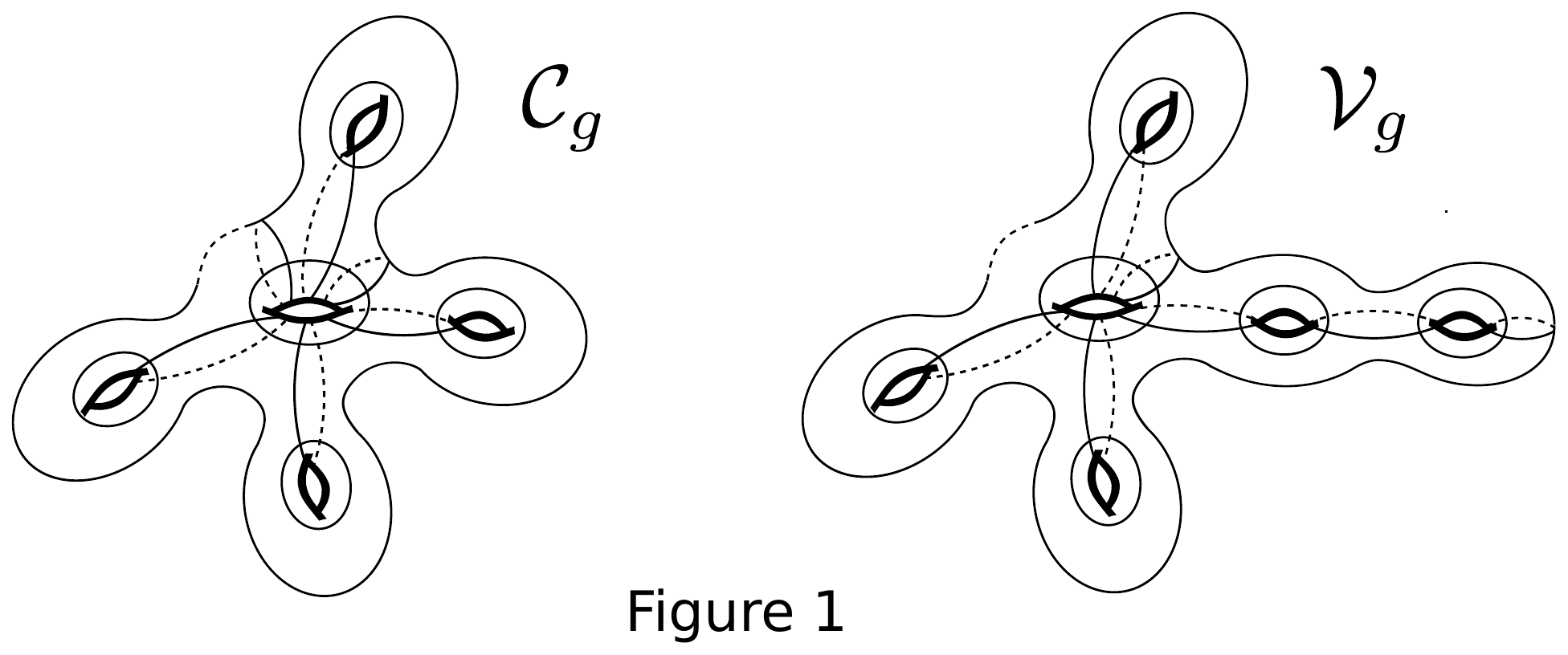}
\end{center}
\end{figure}
We show

\begin{theo}\label{main2}
\begin{enumerate}
\item 
Let $\phi$ be an odd $\mathbb{Z}/2 \mathbb{Z}$-spin structure on 
a closed surface $\Sigma_g$ of genus $g\geq 3$. Then 
${\rm Mod}(\Sigma_g)[\phi]$ is generated by the Dehn twists about the curves from the
curve system ${\cal C}_g$. 
\item Let $\phi$ be an even $\mathbb{Z}/2\mathbb{Z}$-spin structure
on a closed surface $\Sigma_g$ of genus $g\geq 4$. Then 
${\rm Mod}(\Sigma_g)[\phi]$ is generated by the Dehn twists about the curves from the
curve system ${\cal V}_g$. 
\end{enumerate}
\end{theo}

That the spin mapping class group can be generated by finitely many 
finite products of Dehn twists 
is due to Hirose. In \cite{Hi02} he found for any genus $g\geq 2$ 
a generating set for the stabilizer of 
an even $\mathbb{Z}/2\mathbb{Z}$-spin structure by finitely many finite
products of Dehn twists, and the stabilizer of 
an odd $\mathbb{Z}/2\mathbb{Z}$-spin structure is treated in \cite{Hi05}. 

For surfaces of genus $g\geq 5$, 
Calderon \cite{Cal19} and Calderon and Salter \cite{CS19} 
identified the image of the
orbifold fundamental group of most components of strata 
in the mapping class group by constructing a different
but equally explicit generating
set for the spin mapping class group.  Earlier Salter 
(Theorem 9.5 of \cite{Sa19}) obtained 
a partial result by identifying for $g\geq 5$ a finite generating set 
of a finite index subgroup of the spin mapping class group by Dehn twists.
Walker \cite{W09,W10} obtained
some information on the image of the orbifold fundamental group of 
some strata of quadratic differentials in the mapping class group
using completely different tools.  

Theorem \ref{main2} does not construct generators for
the stabilizer of an even $\mathbb{Z}/2\mathbb{Z}$-spin
structure on a surface of genus $g=2,3$. Namely,
in these cases there is no admissible curve system
with the property that the Dehn twists about the curves
from the system stabilize an even
$\mathbb{Z}/2\mathbb{Z}$-spin structure and 
such that the Dehn twists about these curves generate a finite
index subgroup of the mapping class group.
This corresponds to a classification result of Kontsevich and
Zorich \cite{KZ03}: There is no component of a stratum
of abelian differentials with a single zero on a surface of genus $2$
and even spin structure. On a surface $\Sigma_3$ of genus 3,
the component of the stratum of 
abelian differentials with two zeros of order
two and even spin structure is hyperelliptic and hence
the projection of its orbifold fundamental group to
${\rm Mod}(\Sigma_3)$ 
commutes with a hyperelliptic involution and 
is of infinite index.

Our results can be used to construct an
explicit finite set of generators
of the stabilizer of a $\mathbb{Z}/r\mathbb{Z}$-spin structure
for any $r\leq 2g-2$ and any closed surface $\Sigma_g$, 
given by Dehn twists, 
positive powers of Dehn twists and products of Dehn twists about two
simple closed curves forming a bounding pair.
Potentially 
they can also be used inductively to find generators 
by Dehn twists about curves from an admissible curve system.  
We carry this program only out in a single case,
which is the odd $\mathbb{Z}/4\mathbb{Z}$-spin structure on a
surface of genus 3.

Consider the system ${\cal E}_6$ of simple closed curves on the
surface $\Sigma_3$ of genus 3
shown in Figure 2 which is of particular relevance for
the understanding of the stratum of abelian differentials with a single
zero on $\Sigma_3$ \cite{LM14}. 
\begin{figure}[ht]
\begin{center}
\includegraphics[width=0.8\textwidth]{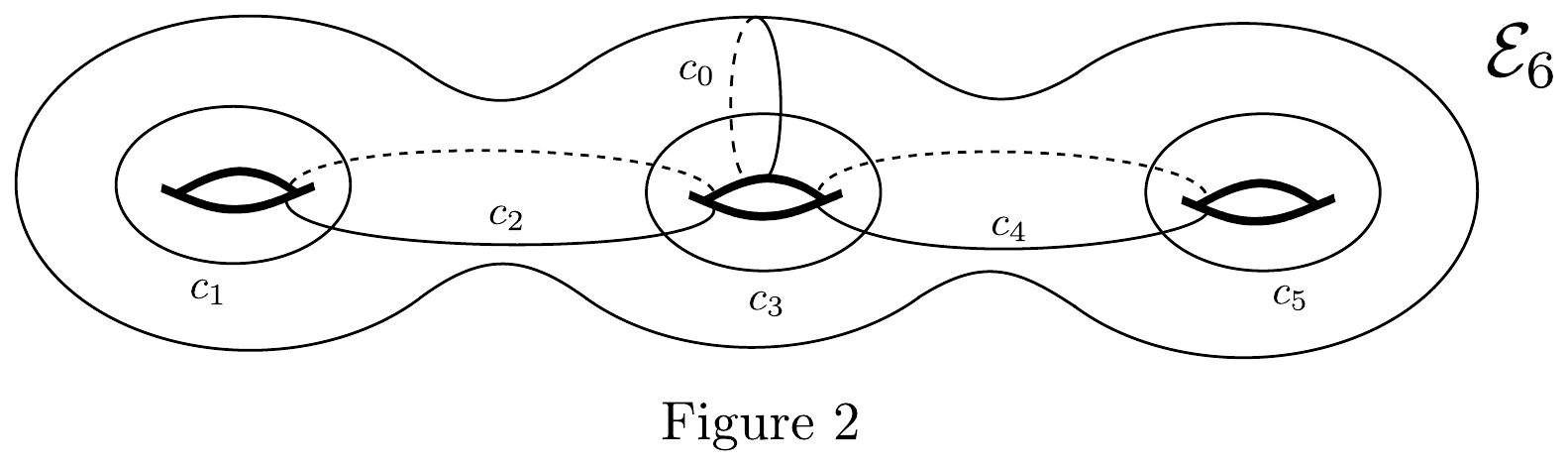}
\end{center}
\end{figure}
We show

\begin{theo}\label{main3}
The subgroup of ${\rm Mod}(\Sigma_3)$ generated by the Dehn twists about the curves
from the curve system ${\cal E}_6$ equals the stabilizer of an odd 
$\mathbb{Z}/4\mathbb{Z}$-spin structure on $\Sigma_3$.
\end{theo}

The strategy for the
proofs of the  main results 
is as follows. 

For some $r\geq 2$ 
let us consider an arbitrary $\mathbb{Z}/r\mathbb{Z}$-spin structure 
$\phi$ on a compact oriented
surface $S$ of genus $g\geq 2$,  perhaps with boundary. 
Following \cite{HJ89} and \cite{Sa19}, 
the spin structure can be viewed as a
$\mathbb{Z}/r\mathbb{Z}$-valued function on oriented closed curves on $S$
which 
assumes the value one on the
oriented boundary of an embedded disk in $S$. 
Changing the orientation of the curve changes the value of $\phi$ on the curve to its negative
\cite{HJ89,Sa19}.

Define a graph ${\cal C\cal G}_1^+$ as follows. Vertices are nonseparating simple 
closed curves $c$ on $S$ with $\phi(c)=\pm 1$,
and two such vertices $d,e$  are 
connected by an edge if $d,e$ can be realized disjointly and if furthermore, 
$S-(d\cup e)$ is connected. Thus ${\cal C\cal G}_1^+$ is a subgraph of the curve
graph of $S$. The stabilizer ${\rm Mod}(S)[\phi]$ of $\phi$ in the mapping class group
of $S$ acts on ${\cal C\cal G}_1^+$ as a group of simplicial automorphisms.

In Section \ref{graphsofcurves} we show
that for any $g\geq 3$ and $r\leq 2g-2$ 
the graph ${\cal C\cal G}_1^+$ is connected.
We also note that for an odd $\mathbb{Z}/2\mathbb{Z}$-spin 
structure on a surface of genus $g=2$, this 
is not true. In 
Section \ref{theaction} we verify that the action of the group 
${\rm Mod}(S)[\phi]$ on the graph ${\cal C\cal G}_1^+$ is transitive 
on vertices.

For a vertex $c$ of ${\cal C\cal G}_1^+$ we are then led to
  describing the intersection of ${\rm Mod}(S)[\phi]$ with the
  stabilizer of $c$ in ${\rm Mod}(S)$.
Most important is the understanding of the intersection of
${\rm Mod}(S)[\phi]$ with the so-called \emph{disk pushing subgroup},
namely the kernel of the natural homomorphism of the stabilizer of
$c$ to the mapping class group of the surface obtained from
$S-c$ by capping off the two distinguished boundary components of
$S-c$. This is also carried out in Section \ref{theaction}.

In 
Section \ref{structureof} we specialize further to 
a $\mathbb{Z}/2\mathbb{Z}$-spin structure $\phi$. We find a presentation of  
${\rm Mod}(S)[\phi]$ as a quotient of a 
$\mathbb{Z}/2\mathbb{Z}$-extension of the free product
of two copies of the stabilizer of a vertex of ${\cal C\cal G}_1^+$,
amalgamated over the stabilizer of an edge of ${\cal C\cal G}_1^+$.
This is used to prove the first part of 
Theorem \ref{main2} 
with an argument by induction on the genus $g$ of the
closed surface $\Sigma_g$.

The proof of the second part of Theorem \ref{main2} 
uses similar methods and 
is contained in 
Section \ref{structureeven}. A variation of these arguments 
yield the proof of Theorem \ref{main3} in  Section \ref{special}.

The appendix contains a technical variation of the main result of
Section \ref{graphsofcurves} which is used in Section \ref{structureeven}.
Its proof follows along exactly the same line as the proof of the main
result of Section \ref{graphsofcurves}.

This work is inspired by the article \cite{Sa19} of Salter. 
However, aside from some simple constructions using curves
and Proposition 4.9 of \cite{Sa19}, our approach 
uses different methods.

{\bf Acknowledgement:} I am grateful to Dawei Chen,
Samuel Grushevsky, Martin M\"oller and 
Nick Salter for useful discussions. Thanks to Susumu Hirose 
for pointing out the references \cite{Hi02} and \cite{Hi05}. 
Finally I am very indebted to the anonymous referee for careful reading
and for suggesting the proof of Proposition \ref{genus3} which largely
simplifies my original argument. 
This work was completed while the author was in 
residence at the MSRI in Berkeley, California, in the fall
semester 2019, supported by the National Science Foundation
under Grant No. DMS-1440140.

\section{Graphs of curves with fixed spin value}\label{graphsofcurves}

In this section we consider a compact surface $S$ of genus
$g\geq 2$, with or without boundary. For a number $r\geq 2$ 
we introduce $\mathbb{Z}/r\mathbb{Z}$-spin  structures on 
$S$ and use these structures to define various subgraphs of the 
curve graph of $S$. Of primary interest is a graph ${\cal G\cal G}_1$ 
whose vertices are 
nonseparating simple closed curves
with spin value $\pm 1$ and where two such curves are connected
by an edge if they can be realized disjointly. We then study connectedness of this graph.

Small genus of the surface may cause the graph ${\cal C\cal G}_1$  
to have few edges.
This problem leads us to proceed in two steps. In Proposition \ref{genus3} we 
 show connectedness of ${\cal C\cal G}_1$ for surfaces of genus $g\geq 3$ and
 $r=2,4$, taking advantage of some special properties of $\mathbb{Z}/2\mathbb{Z}$ and
 $\mathbb{Z}/4\mathbb{Z}$ spin structures. Proposition \ref{connected6}
 shows connectedness of ${\cal C\cal G}_1$ for surfaces of genus $g\geq 4$ and
 all $r$, taking advantage of sufficiently large complexity of the underlying surface. 
These results are used in Section \ref{theaction} 
to study the stabilizer of a spin structure in the mapping class group of $S$.

This section is divided into 5 subsections. We begin with summarizing some
information on spin structures. Each of the remaining subsections is devoted
to the investigation of a specific subgraph of the curve graph of $S$ defined
by a spin structure $\phi$ on $S$.

\subsection{Spin structures}

The following is taken from
\cite{HJ89}, see Definition 3.1 of \cite{Sa19}. For its formulation,
denote by $\iota$ the symplectic form on $H_1(S,\mathbb{Z})$. 

\begin{definition}[Humphries-Johnson]\label{spin}
For a number $r\geq 2$, a 
\emph{$\mathbb{Z}/r\mathbb{Z}$-spin structure} on $S$ 
is a $\mathbb{Z}/r\mathbb{Z}$-valued 
function $\phi$ on 
isotopy classes of oriented simple closed
curves on $S$ with the following properties.
\begin{enumerate}
\item (Twist linearity) Let $c,d$ be oriented simple closed curves
and let $T_c$ be the left Dehn twist about $c$; then
\[\phi(T_c(d))=\phi(d)+\iota(d,c)\phi(c)\quad \text{ (mod }r).\]
\item (Normalization) $\phi(\zeta)=1$ for 
the oriented boundary $\zeta$ of an embedded disk
$D\subset S$.
\end{enumerate}
\end{definition}

As an additional property, one obtains that 
whenever 
$c^{-1}$ is obtained from $c$ by reversing the orientation,
then $\phi(c^{-1})=-\phi(c)$ (Lemma 2.2 of \cite{HJ89}).

Humphries and Johnson \cite{HJ89} (see Theorem 3.5 of \cite{Sa19})
also give an alternative description of spin structures.
Namely, for some choice of a hyperbolic metric on $S$ let
 $UTS$ be the unit tangent bundle of $S$. 
  It can be viewed as the quotient of the complement of the
  zero section in the tangent bundle of $S$ by the multiplicative
  group $(0,\infty)$ and hence it does not depend on the metric.

The \emph{Johnson lift} of a
smoothly embedded oriented simple closed curve $c$ on $S$ 
is simply the closed curve in $UTS$ which consists of all unit tangents of $c$
defining the given orientation. The following is Theorem 2.1 and Theorem 2.5 of 
\cite{HJ89} as formulated in Theorem 3.5 of \cite{Sa19}. 

\begin{theorem}[Humphries-Johnson]\label{cohomology}
Let $S$ be a compact surface and let $\zeta$ be the oriented
fibre of the unit tangent bundle $UTS\to S$. 
A cohomology class $\psi\in H^1(UTS,\mathbb{Z}/r\mathbb{Z})$
with $\psi(\zeta)=1$ 
determines a $\mathbb{Z}/r\mathbb{Z}$-spin structure via
\[\alpha\to \psi(\tilde \alpha)\]
where $\alpha$ is an oriented simple closed curve on $S$ and
$\tilde \alpha$ is its Johnson lift. This determines a 1-1 correspondence
between $\mathbb{Z}/r\mathbb{Z}$-spin structures and 
\[\{\psi\in H^1(UTS,\mathbb{Z}/r\mathbb{Z})\mid \psi(\zeta)=1\}.\]
\end{theorem}

There is another interpretation as follows; we refer to p.131 of \cite{H95}
for more information on this construction. 
Given a number $r\geq 2$ which 
divides $2g-2$, 
an application of the Gysin sequence for the Euler class of $UTS$ yields
a short exact sequence
\begin{equation}\label{gysin}
0\to \mathbb{Z}/r\mathbb{Z}\to H_1(UTS,\mathbb{Z}/r\mathbb{Z})\to 
H_1(S,\mathbb{Z}/r\mathbb{Z})\to 0.\end{equation}
By covering space theory, 
an $r$-th root of the tangent bundle of $S$, viewed as a complex line
bundle for some fixed complex structure, is determined by a homomorphism
$H_1(UTS,\mathbb{Z}/r\mathbb{Z})\to  \mathbb{Z}/r\mathbb{Z}$
whose composition with the inclusion 
$\mathbb{Z}/r\mathbb{Z}\to H_1(UTS,\mathbb{Z}/r\mathbb{Z})$ is the identity
and therefore

\begin{proposition}\label{rspin}
There is a natural one-to-one correspondence between the $r$-th roots
of the canonical bundle of $S$ and splittings of the sequence 
(\ref{gysin}).
\end{proposition}

A $\mathbb{Z}/2\mathbb{Z}$-spin structure on a compact 
surface $S$ of genus $g$ with empty or connected boundary
has a \emph{parity} which is defined  
as follows.

A \emph{geometric symplectic basis} for $H_1(S,\mathbb{Z})$ 
is a system
$a_1,b_1,\dots,a_g,b_g$ of simple closed curves
on $S$ such that $a_i,b_i$ intersect in a single point and that
$a_i\cup b_i$ is disjoint from $a_j\cup b_j$ for $i\not=j$. 
Then the parity of the spin structure $\phi$ equals
\begin{equation}\label{arf}
{\rm Arf}(\phi)
=\sum_i(\phi(a_i)+1)(\phi(b_i)+1)
  \in \mathbb{Z}/2\mathbb{Z}.\end{equation}
This does not depend on the choice of the geometric symplectic basis.

\subsection{The graph of nonseparating curves with vanishing spin value}

 The \emph{curve graph} ${\cal C\cal G}$ of $S$ is the graph whose vertices
 are \emph{essential} (that is, neither nullhomotopic nor homotopic into
 the boundary) 
simple closed curves in $S$ and where two such curves are connected by an edge
if they can be realized disjointly. We can use the spin structure $\phi$ 
 to introduce various subgraphs of ${\cal C\cal G}$ and study their properties. 
 One of the main technical ingrediences to this end is the following result
 of Salter (Corollary 4.3 of \cite{Sa19}). 
 
 \begin{lemma}[Salter]\label{torus}
  Let $\Sigma\subset S$ be an embedded one-holed torus. Then there exists
  a simple closed curve $c\subset \Sigma$ with $\phi(c)=0$. 
\end{lemma}  



Denote by ${\cal C\cal G}_0\subset {\cal C\cal G}$
the complete subgraph of the
curve graph whose vertex set consists of nonseparating curves $c$ with 
$\phi(c)=0$. Note that this is well defined, that is, it is independent of the choice
of an orientation of $c$. 
As a fairly easy consequence of Lemma \ref{torus}
we obtain

\begin{lemma}\label{connected0}
Let $\phi$ be a spin structure on a closed surface of genus $g\geq 3$.
Then ${\cal C\cal G}_0$ is connected.
\end{lemma}
\begin{proof} We use the following result of Masur-Schleimer
  \cite{MS06}, see Theorem 1.2 of 
\cite{Put08}. Let ${\cal S\cal G}\subset
  {\cal C\cal G}$ be the complete subgraph whose vertex set
  consists of \emph{separating} simple closed curves; then 
${\cal S\cal G}$ is connected. Note that this requires that $g\geq 3$.

Let $a,b$ be vertices of ${\cal C\cal G}_0$. Choose simple closed curves 
$\hat a,\hat b$ which intersect $a,b$ in a single point; such curves exist since $a,b$ are 
nonseparating. Then the boundary $c,d$ of a tubular neighborhood of
$a\cup\hat  a$ and $b\cup \hat b$, respectively, 
is a separating simple closed curve which decomposes $S$ into 
a one-holed torus containing $a,b$ and a surface of genus $g-1\geq 2$ with 
boundary.

Connect $c$ to $d$ by an edge path $(c_i)_{0\leq i\leq k}\subset {\cal S\cal G}$
(here $c=c_0$ and $d=c_k$).  
Construct inductively an 
edge path $(a_i)\subset {\cal C\cal G}_0$ connecting $a=a_0$ to $b=a_k$
such that
for each $i$, $a_{i}$ is disjoint from $c_i$, as follows.
Put $a_0=a$ and 
assume that we constructed already such 
a path for some $j<k$. Then $a_{j}$ is disjoint from $c_j$. 

If $a_{j}$ also is disjoint from $c_{j+1}$ then define $a_{{j+1}}=a_{j}$. Otherwise
$a_{j}$ is contained in the same component
$\Sigma$ of $S-c_j$ as $c_{j+1}$. 
Choose a one-holed torus $T\subset S-\Sigma$. Such a torus exists since
$c_j$ decomposes $S$ into two surfaces of positive genus with connected boundary.  
By Lemma \ref{torus},
this torus contains a 
nonseparating simple closed curve $a_{j+1}$ 
with $\phi(a_{j+1})=0$, and this curve is disjoint from
both $a_{j}$ and $c_{j+1}$.  This yields the induction step.
\end{proof}

\begin{remark}\label{genus2rm}
The proof of Lemma \ref{connected0} extends with a bit more care to compact
surfaces of genus at least 3 with connected boundary. 
We expect that the Lemma also holds true for $g=2$.
  \end{remark}

\subsection{The graph of nonseparating 
curves with spin value $\pm 1$ on a surface of genus $2$}

 Define 
 ${\cal C\cal G}_1$ to be the complete subgraph of ${\cal C\cal G}$
 of all nonseparating simple closed 
 curves $c$ on $S$ with $\phi(c)=\pm 1$. Note that this
 condition does not depend on the orientation of $c$
 and hence it is indeed a condition on the vertices of 
 ${\cal C\cal G}$.
In this subsection we discuss the special case $g=2$.

\begin{proposition}\label{g=2}
Let $\phi$ be an odd $\mathbb{Z}/2\mathbb{Z}$-spin structure 
on a closed surface $S$ of genus $2$. Then  
any two simple closed nonseparating 
curves $c,d$ on $S$ with $\phi(c)=\phi(d)=1$
intersect.
\end{proposition}  
\begin{proof} Let $\phi$ be a $\mathbb{Z}/2\mathbb{Z}$-spin structure on
  $S$. Let 
  $c$ be a nonseparating simple closed curve
  on $S$ with $\phi(c)=1$. 
  Assume that there is a nonseparating simple closed curve
  $d$ with $\phi(d)=1$ which is disjoint from $c$.
  As a surface of genus two does not admit bounding pairs,
  the surface $S-(c\cup d)$ is a four-holed sphere. Thus there exists
  a simple closed separating 
  curve $e$ which decomposes $S$ into two one-holed tori $T_1,T_2$ such that
  $c\in T_1,d\in T_2$.

  Denoting by $\iota$ the mod two homological intersection
  form on $H_1(S,\mathbb{Z}/2\mathbb{Z})$, 
  there are two nonseparating simple closed curves $v\subset T_1,w\subset T_2$ 
  so that
  \begin{equation}\label{intersectionrel}
    \iota(v,c)=1=\iota(w,d)\text{ and }\iota(w,c)=\iota(v,d)=0.
  \end{equation}
  
  The curves $a_1=c,b_1=v,
  a_2=d,b_2=w$ define a geometric symplectic basis 
  for $H_1(S,\mathbb{Z})$. Since $\phi(a_1)=\phi(a_2)=1$, the formula
  (\ref{arf}) for the Arf invariant shows that $\phi$ is even as claimed.
   \end{proof} 


\subsection{$\mathbb{Z}/r\mathbb{Z}$-spin structures 
for $r=2,4$ on 
a surface of genus $g\geq 3$}

In this subsection we study the graph ${\cal C\cal G}_1$ 
for a $\mathbb{Z}/r\mathbb{Z}$-spin structure 
on a surface of genus $g\geq 3$ for $r=2,4$. 
We begin with evoking a result of Salter \cite{Sa19}.

Namely, let $c,d$ be disjoint simple closed curves on the compact
surface $S$. Let $\epsilon$ be an embedded arc in $S$ connecting
$c$ to $d$ whose interior is disjoint from $c\cup d$. 
A regular neighborhood $\nu$ of 
$c\cup \epsilon \cup d$ is homeomorphic to a three-holed
sphere. Two of the boundary components of $\nu$ are
the curves $c,d$ up to homotopy. We choose an orientation
of $c,d$ in such a way that $\nu$ lies to the left. 
The third boundary component $c +_\epsilon d$, oriented in such a
way that $\nu$ is to its right, satisfies 
$[c +_\epsilon d]=[c]+[d]$ where
$[c]$ denotes the homology class of the oriented curve $c$. 
The following is Lemma 3.13 of \cite{Sa19}. 

\begin{lemma}[Salter]\label{add}
$\phi(c+_\epsilon d)=
\phi(c)+\phi(d)+1$.
\end{lemma}

As a consequence, if $r=2,4$ then the boundary 
of any embedded pair of pants
$P\subset S$ contains a simple closed curve $c$ with 
$\phi(c)=\pm 1$. To use this fact for our purpose 
 we introduce another graph related to
 simple closed curves on surfaces.

   \begin{definition}\label{nonseparatingpairs}
Let $S$ be a compact surface of genus $g\geq 2$. 
The \emph{graph of non\-separating pairs of pants} ${\cal N\cal S}$ is the graph whose
   vertices are pairs of pants in $S$ whose boundary consists of three pairwise distinct
    nonseparating simple closed curves 
   and where two such pair of pants are connected by an edge if their intersection
   consists of precisely one boundary component. 
\end{definition}
 
By the preceding remark, the graph ${\cal N\cal S}$ can be used to find paths
in the graph ${\cal C\cal G}_1$ provided we can show that it is connected.
To this end we evoke an observation of Putman
(Lemma 2.1 of \cite{Put08})
which we refer to as
the \emph{Putman trick} in the sequel. 

\begin{lemma}[Putman]\label{putmantrick}
  Let $G$ be a graph which admits a vertex transitive isometric
  action of a finitely generated group $\Gamma$ and let
  $v$ be a vertex of $G$. If 
    for each element $s$
    of a finite generating set ${\cal S}$ of $\Gamma$,
    the vertex $v$ can be connected
    to $sv$ by an edge path in $G$, then $G$ is connected.
\end{lemma}

We apply the Putman trick to show

\begin{lemma}\label{connectedpants}
The graph of nonseparating pairs of pants is connected.
\end{lemma}
\begin{proof} If $P\subset S$ is a nonseparating pair of pants,
then $S-P$ is a connected surface of genus $g-2$ with three distinguished
boundary components. Thus the pure mapping class group 
$P{\rm Mod}(S)$ 
of $S$ acts transitively on the vertices of ${\cal N\cal S}$. 
As a consequence, it suffices to show that 
there exists a generating set ${\cal S}$ of $P{\rm Mod}(S)$ and 
a nonseparating pair of pants $P\in {\cal N\cal S}$ which  
can be connected to its image $\psi(P)$
by an edge path in ${\cal N\cal S}$ 
for every element $\psi\in {\cal S}$.

Now $P{\rm Mod}(S)$ can be generated by Dehn twists $T_{c_i}$ 
about the collection of simple closed curves $c_0,\dots,c_k$
shown in Figure 3 
(see Section 4.4 of \cite{FM12}).\\
\begin{figure}[ht]
\begin{center}
\includegraphics[width=0.9\textwidth]{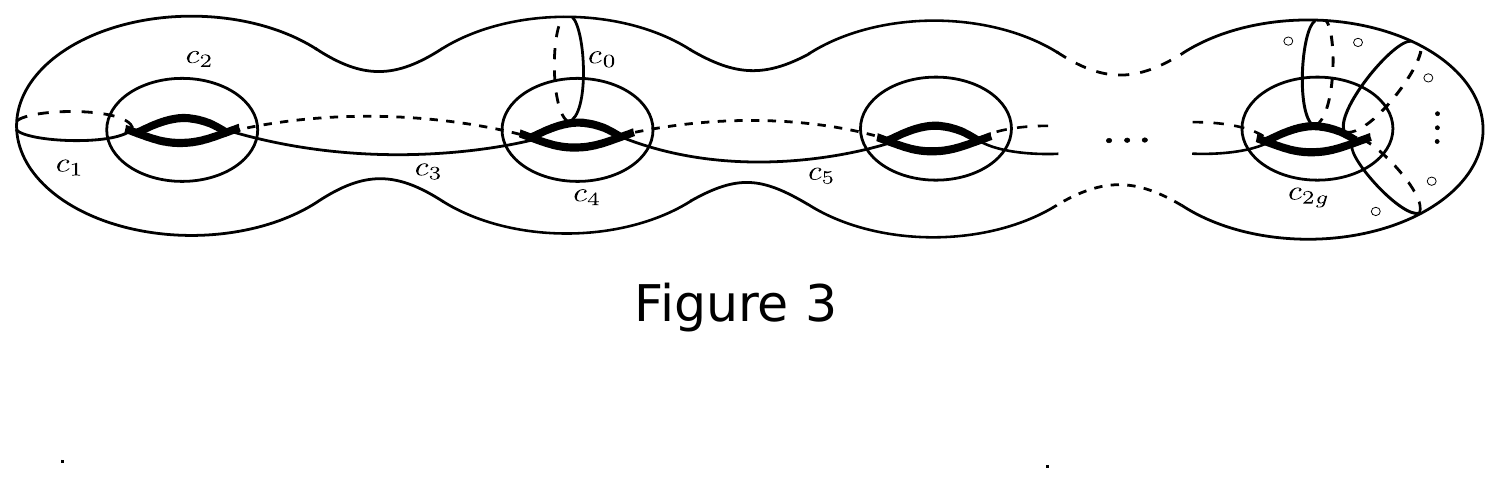}
\end{center}
\end{figure}
Furthermore, the simple closed curves $c_0,c_1,c_3$ 
are nonseparating and bound a pair of pants $P$. This pair of pants
is stabilized by all elements of ${\cal S}$ with the exception of the 
Dehn twists about the simple closed curves $c_2$ and $c_4$. 

As the genus of $S$ is at least $3$, for $i=2,4$ 
the complement $\hat S$ 
in $S$ of the union of $P$ with $c_i$ is a surface of genus
$g-2\geq 1$ with two distinguished boundary components, 
one of which, say the curve $c$, 
is a boundary component of $P$.

The surface $\hat S$ contains a nonseparating simple closed curve 
$d$. As in Lemma \ref{add}, choose an embedded arc $\epsilon\subset \hat S$
connecting $c$ to $d$. A small neighborhood of 
$c\cup \epsilon \cup d$ is a nonseparating pair of pants $\hat P$ whose intersection 
with $P$ equals the curve $c$.
As $\hat P$ is disjoint from $c_i$, it is 
left fixed by the Dehn twist $\psi_i$ about $c_i$ and hence 
$P,\hat P,\psi_i(P)$ is a path of length three in ${\cal N\cal S}$ connecting
$P$ to $\psi_i(P)$. Lemma \ref{connectedpants} now is an immediate
consequence of Lemma \ref{putmantrick}. 
\end{proof}

The following is now an easy consequence of Lemma \ref{connectedpants}.
 
\begin{proposition}\label{genus3}
  Let $r=2,4$ and let $\phi$ be a $\mathbb{Z}/r\mathbb{Z}$
  spin structure on a compact surface $S$ of genus $g\geq 3$, with or without boundary. 
  Then the graph ${\cal C\cal G}_1$ is connected.
\end{proposition}  
 \begin{proof}
Let $c,d$ be nonseparating simple closed curves with $\phi(c)=\pm 1, \phi(d)=\pm 1$.
Choose nonseparating pairs of pants $P,Q$ containing $c,d$ in their boundary. 
By Lemma \ref{connectedpants}, we can connect $P$ to $Q$ by a path in the graph
${\cal N\cal S}$, say the path $(P_i)$ with $P_0=P$ and $P_k=Q$. 

For each $i$ let $c_i$ be a boundary component of $P_i$ with $\phi(c_i)=\pm 1$ and such that
$c_0=c,c_k=d$. Then for each $i$, either $c_i=c_{i+1}$ or $c_i$ and $c_{i+1}$ are disjoint. 
Thus $(c_i)$ is a path in ${\cal C\cal G}_1$ connecting $c$ to $d$. This shows the proposition.
 \end{proof}

\subsection{$\mathbb{Z}/r\mathbb{Z}$-spin structures on a surface of genus 
$g\geq 4$} \label{allr}

In this subsection we investigate the graph ${\cal C\cal G}_1$ on a surface
of genus $g\geq 4$ for an arbitrary $r\geq 2$. To show connectedness
we use the following auxiliary graph 
${\cal P\cal S}$. The vertices of
${\cal P\cal S}$ are pairs of disjoint separating curves
$(c,d)$ which each decompose $S$ into a surface of genus 
$g-1$ and a one-holed torus.
Thus $S-(c\cup d)$ is the disjoint union of two
one-holed tori and a surface of genus $g-2$. 
Two such pairs $(c_1,d_1)$ and $(c_2,d_2)$
are connected by an edge if up to renaming,
$c_1=c_2$ and $d_2$ is disjoint from $c_1,d_1$. Then
$S-(c_1\cup d_1\cup d_2)$ is the disjoint union of a surface of
genus $g-3$ with at least three holes and three one-holed tori. In particular,
the graph ${\cal P\cal S}$ is only defined if the genus of $S$ is at least three.

We use the Putman trick to show

\begin{lemma}\label{connected5}
  For a compact surface $S$ of genus $g\geq 4$, perhaps with boundary, 
 the graph ${\cal P\cal S}$ is a connected 
  ${\rm Mod}(S)$-graph.
\end{lemma}  
\begin{proof} The mapping class group ${\rm Mod}(S)$ of the surface $S$ 
clearly acts
  on ${\cal P\cal S}$, furthermore this action is vertex transitive.
  Namely, for any two vertices $(a_1,b_1)$ and $(a_2,b_2)$
  of ${\cal P\cal S}$, the complement
  $S-(a_i\cup b_i)$ is the union of two one-holed tori and a surface
  of genus $g-2$ with $k+2$ boundary components 
where $k\geq 0$ is the number of boundary components of $S$.  
  Hence 
  there exists $\phi\in {\rm Mod}(S)$ with $\phi(a_1,b_1)=(a_2,b_2)$.

  Consider again the curve system ${\cal H}$ 
  shown in Figure 3 with the
  property that the Dehn twists about these curves
  generate the mapping class group. Choose a pair of
  disjoint separating simple closed curves
  $(a,b)$ which decompose $S$ into a surface of genus
  $g-1$ and a one-holed torus $X(a), X(b)$ and such that 
  a curve $c\in {\cal H}$ intersects at most one of the
  curves $a,b$. If it intersects
  one of the curves $a,b$, then this intersection
   consists of precisely two points.
   For example, we can choose $a$ to be the
   boundary of a small neighborhood of $c_1\cup c_2$, and 
   $b$ to be the boundary of a small neighborhood of $c_5\cup c_6$.
   
   Now let $c\in {\cal H}$ and let $T_c$ be the left Dehn twist about
   $c$. If $c$ is disjoint from $a\cup b$, then $T_c(a,b)=(a,b)$ and
   there is nothing to show. Thus assume that $c$ intersects $a$.
   
   The image $T_c(a)$ of $a$ is a separating simple closed curve
   contained in a small neighborhood $Y$ of $X(a)\cup c$. 
   By assumption on $c$,
   this surface is a two-holed torus disjoint from $b$. 
   As $g\geq 4$, the genus of $S-(Y\cup X(b))$ is at least one
   and hence there is a separating curve
   $e\subset S-(Y\cup X(b))$ 
   which decomposes $S-(Y\cup X(b))$ into a one-holed
   torus and a surface $S^\prime$. But this means that
   $(a,b)$ can be connected to $T_c(a,b)=(T_ca,b)$ by the edge path
   $(a,b)\to (e,b)\to (T_ca,b)$. As the roles of $a$ and $b$ can
   be exchanged, the lemma now follows from 
   the Putman trick.
   \end{proof}  

We shall use another auxiliary graph which is defined as follows.

\begin{definition}\label{nonseparc}
Let $S$ be 
a compact surface $S$ of genus $g\geq 1$ with two distinguished
boundary components $A_1,A_2$. The \emph{nonseparating arc graph}
is the graph whose vertices are isotopy classes of 
embedded arcs in $S$ connecting 
$A_1$ to $A_2$. The endpoints of an arc may move freely along the boundary
circles $A_1,A_2$ in such an isotopy class. 
Two such arcs $\epsilon_1,\epsilon_2$ 
are connected by an edge if 
$\epsilon_1,\epsilon_2$ are disjoint and $S-(\epsilon_1\cup \epsilon_2)$ is connected.
\end{definition}

We apply the Putman trick to show

\begin{lemma}\label{connectedarc}
The nonseparating arc graph ${\cal A}(A_1,A_2)$ on a 
compact surface $S$ of genus $g\geq 1$ with two distinguished
boundary components $A_1,A_2$ is connected.
\end{lemma}
\begin{proof}
Clearly the pure mapping class group $P{\rm Mod}(S)$ 
of $S$ acts transitively on 
the vertices of ${\cal A}(A_1,A_2)$, so it suffices to show that 
there exists a generating set ${\cal S}$ of $P{\rm Mod}(S)$ and 
an arc $\epsilon\in {\cal A}(A_1,A_2)$ which  
can be connected to its image $\psi(\epsilon)$
by an edge path in ${\cal A}(A_1,A_2)$ 
for every element $\psi\in {\cal S}$.

There exists two disjoint arcs $\epsilon_1,\epsilon_2$ connecting
$A_1$ to $A_2$ such that $\epsilon_1\cup \epsilon_2$ projects to 
an essential 
nonseparating simple closed curve in the surface obtained from $S$
by capping off the boundary components $A_1,A_2$. Furthermore,
we may assume that 
$\epsilon_1$ intersects one of the curves shown in Figure 3, say the
curve $c_1$, in a single point and is disjoint from the remaining
curves, and $c_1$ is disjoint from $\epsilon_2$.

Then $T_{c_i}\epsilon_1=\epsilon_1$ for $i\geq 2$, and 
$\epsilon_1$ can be connected to $T_{c_1}(\epsilon_1)$ 
by the edge path 
$\epsilon_1,\epsilon_2,T_{c_1}\epsilon_1$. By the Putman trick
this implies that ${\cal A}(A_1,A_2)$ is connected.
\end{proof}

We are now ready to show

\begin{proposition}\label{connected6}
Let $\phi$ be an $r$-spin structure $(r\geq 2)$ 
on a compact surface 
$S$ of genus $g\geq 4$. Then 
the graph ${\cal C\cal G}_1$ is connected.
\end{proposition}  
\begin{proof}
Let $S$ be a compact surface of genus $g\geq 2$ and 
  consider the graph ${\cal P\cal S}$. To each of its vertices,
  viewed as a disjoint pair 
  $(c,d)$ of separating simple closed curves, 
  we associate in a non-deterministic way a vertex 
  $\Lambda(c,d)$ of ${\cal C\cal G}_1$ as follows. 

Denote by $\Sigma_c, \Sigma_d$ the one-holed torus bounded by $c,d$.
If one of the tori $\Sigma_c,\Sigma_d$
contains a simple closed curve $a$ with $\phi(a)=\pm 1$ then 
define $\Lambda(c,d)=a$.

Now assume that
none of the tori $\Sigma_c,\Sigma_d$ contains a simple closed
curve $a$ with $\phi(a)=\pm 1$. By Lemma \ref{torus}, 
  there are simple closed nonseparating curves
  $a\subset \Sigma_c,b\subset \Sigma_d$ so that $\phi(a)=0=\phi(b)$.
  Since the tori $\Sigma_c,\Sigma_d$ are disjoint,
  the pair $(a,b)$ is nonseparating, that is, $S-(a\cup b)$ is connected.
  Choose an embedded arc $\epsilon$ in $S$ 
  connecting $a$ to $b$.
  By Lemma \ref{add},
  the curve $\Lambda(c,d)=a+_\epsilon b$
  satisfies $\phi(a+_\epsilon b)=\pm 1$, furthermore it is nonseparating.

  Let $a$ be any vertex of ${\cal C\cal G}_1$ and let $b$ be any simple closed curve
  which intersects $a$ in a single point. 
  Such a curve exists since $a$ is nonseparating. 
  Then
  a tubular neighborhood of $a\cup b$ is a torus containing $a$.
  Let $c$ be the boundary curve of this torus and choose
  a second separating simple closed curve $d$ so that
  $(c,d)\in {\cal P\cal S}$.

Let $e\in {\cal C\cal G}_1$ be another vertex.
Construct as above a vertex $(p,q)\in {\cal P\cal S}$
so that $e$ is contained in 
the one-holed torus cut out by $p$.
Connect $(c,d)$ to $(p,q)$ by an
edge path $(c_i,d_i)_{0\leq i\leq k}$ in ${\cal P\cal S}$.
We use this edge path to construct an edge path $(a_j)\subset 
{\cal C\cal G}_1$ connecting $a$ to $e$ which passes through
suitable choices $a_{j_i}$ $(i\leq k)$ 
of the curves $\Lambda(c_i,d_i)$.

Define $a_0=a$ and by induction, 
let us assume that we constructed already the path $(a_j)_{0\leq j\leq j_i}$ 
for some $i\geq 0$.
We distinguish two cases.

{\sl Case 1:} One of the tori $\Sigma_{c_i},\Sigma_{d_i}$ contains
a curve $f$ with $\phi(f)=\pm 1$.

By construction, in this case we may assume by renaming that 
$f=a_{j_i}\subset \Sigma_{c_i}$. 

If $c_i\in \{c_{i+1},d_{i+1}\}$ then define
$a_{j_i+1}=a_{j_{i+1}}=a_{j_i}=\Lambda(c_{i+1},d_{i+1})$
and note that this is consistent with
the requirements for the induction step.

Thus we may assume now that $c_i\not\in \{c_{i+1},d_{i+1}\}$.
If one of the tori
$\Sigma_{c_{i+1}},\Sigma_{d_{i+1}}$, say the torus $\Sigma_{c_{i+1}}$, 
contains a curve $h$ with $\phi(h)=\pm 1$, then
as $\Sigma_{c_i}$ is disjoint from $\Sigma_{c_{i+1}}$,
the curve $h$ is disjoint from $a_{j_i}$ and we can define 
$a_{j_i+1}=h=
a_{j_{i+1}}=\Lambda(c_{i+1},d_{i+1})$.

Thus assume that neither 
$\Sigma_{c_{i+1}}$ nor $\Sigma_{d_{i+1}}$ contains such a curve. 
Since $\Sigma_{c_i}$ and $\Sigma_{c_{i+1}},\Sigma_{d_{i+1}}$ are pairwise
  disjoint, we can find an embedded arc $\epsilon$ in
  $S-\Sigma_{c_i}$ connecting
  a simple closed 
  curve $u\subset \Sigma_{c_{i+1}}$ with $\phi(u)=0$ to a curve
  $h\subset \Sigma_{d_{i+1}}$ with $\phi(h)=0$.  
We then can define $a_{j_i+1}=u+_\epsilon h=\Lambda(c_{i+1},d_{i+1})
=a_{j_{i+1}}$.

{\sl Case 2:} None of the tori $\Sigma_{c_i},\Sigma_{d_i}$ contains
a curve $f$ with $\phi(f)=\pm 1$.

In this case there are simple closed curves
$f\subset \Sigma_{c_i},h\subset \Sigma_{d_i}$
with $\phi(f)=\phi(h)=0$, and there
is an embedded
arc $\epsilon$ connecting $f$ to $h$ so that
\[a_{j_i}=\Lambda(c_{i+1},d_{i+1})=f+_\epsilon h.\] 
Assume without loss of generality that $d_i=d_{i+1}$. 

Let us in addition assume for the moment that 
the arc
$\epsilon$ is disjoint from $c_{i+1}$. If furthermore
there exists a simple closed curve $u\subset
\Sigma_{c_{i+1}}$ with $\phi(u)=\pm 1$, then this curve is
a choice for $\Lambda(c_{i+1},d_{i+1})$
which is disjoint from $a_{j_i}$ and we are done.

Otherwise 
cut $S$ open along the simple closed curve 
$h\subset \Sigma_{d_i}=\Sigma_{d_{i+1}}$ and let $H_1,H_2$
be the two boundary components of $S-h$. By renaming, 
assume without loss
of generality that $\epsilon$ connects the boundary component $H_1$ to 
the curve $f$, i.e. it leaves the curve $h$ from the side
corresponding to $H_1$.
Now note that $M=S-h-\epsilon-\Sigma_{c_i}$ is a
connected surface of genus $g-2\geq 2$ with two distinguished 
boundary
circles, one of which is the curve $H_2$, and 
$M\supset \Sigma_{c_{i+1}}$. Therefore
there exists an embedded arc $\epsilon^\prime\subset M$ connecting
$H_2$ to a simple closed curve $u\subset \Sigma_{c_{i+1}}$ with
$\phi(u)=0$. Define $a_{j_i+1}=h+_{\epsilon^\prime}u$ and note that
this definition is consistent with all requirements.
This construction completes the induction step
under the additional assumption that arc $\epsilon$ is disjoint from
$c_{i+1}$.

We are left with the case that 
$\epsilon$ is \emph{not} disjoint from
$\Sigma_{c_{i+1}}$. Cut $S$ open along $f\cup h$ and note that
the resulting surface $Z$ has genus $g-2\geq 2$ and four 
distinguished boundary components,
say the components $F_1,F_2,H_1,H_2$. Assume that $\epsilon$
connects $F_1$ to $H_1$.

Consider the nonseparating arc  graph ${\cal A}(F_1,H_1)$ 
in $Z$ of arcs connecting 
$F_1$ to $H_1$. By Lemma \ref{connectedarc}, 
this graph is 
connected. Let $\epsilon_i$ be a path in ${\cal A}(F_1,H_1)$ 
which connects $\epsilon$ to an arc $\epsilon^\prime$ 
disjoint from
$\Sigma_{c_{i+1}}$. For any two consecutive of such arcs, say the arcs
$\epsilon_j,\epsilon_{j+1}$, the surface $Z-(\epsilon_1\cup \epsilon_2)$ is connected and 
hence we can find a disjoint arc $\delta_j$ connecting
$F_2$ to $H_2$. The
curves $f+_{\epsilon_j}h,f+_{\delta_j}h,
f+_{\epsilon_{j+1}}h$ are disjoint and yield
a path in ${\cal C\cal G}_1$ 
connecting $f+_\epsilon h$ to a curve $f+_{\epsilon^\prime} h$ which is 
disjoint from $\Sigma_{c_{i+1}}$.
We then can apply the construction for the case that
the arc connecting $f$ to $h$ is disjoint from $\Sigma_{c_{i+1}}$. 
This completes the proof of the proposition.
\end{proof}

For technical reasons we need a stronger version of
Proposition \ref{genus3} and Proposition \ref{connected6}.
Consider a 
$\mathbb{Z}/r\mathbb{Z}$-spin structure $\phi$ 
on a compact surface $S$ of genus $g$ (with or without boundary) 
for an arbitrary number $r\geq 2$. 
We introduce another graph 
${\cal C\cal G}_1^+$ as follows. The vertices of 
${\cal C\cal G}_1^+$ coincide with the vertices
of ${\cal C\cal G}_1$. Any two such
vertices $c,d$ are connected by an edge if $c,d$
are disjoint and if furthermore 
$S-(c\cup d)$ is connected. Thus ${\cal C\cal G}_1^+$ is 
obtained from ${\cal C\cal G}_1$ by removing some of the edges. 
In particular, if ${\cal C\cal G}_1^+$ is connected
then then same holds true for
${\cal C\cal G}_1$. We use connectedness of ${\cal C\cal G}_1$ to establish
connectedness of ${\cal C\cal G}_1^+$. 

\begin{lemma}\label{next}
If the genus $g$ of $S$ is at least 3 then 
the graph ${\cal C\cal G}_1^+$ is connected provided that
  ${\cal C\cal G}_1$ is connected.  
\end{lemma}
\begin{proof}
  Let $c,d\in {\cal C\cal G}_1$
  be two vertices which are connected by an edge in ${\cal C\cal G}_1$
  and which are not connected by an edge in ${\cal C\cal G}_1^+$.
  This means that $c,d$ are disjoint, and $S-(c\cup d)$ is disconnected.
  We have to show that $c,d$ can be connected in ${\cal C\cal G}_1^+$
  by an edge path.

  To this end recall that $c,d$ are nonseparating and therefore
 the disconnected surface 
  $S-(c\cup d)$ has two connected components $S_1,S_2$. The surface 
  $S_1$ has genus 
  $g_1\geq 1$ and at least two boundary components,
  and the surface $S_2$ has genus 
  $g_2=g-g_1-1\geq 0$ and at least two boundary components.

Choose a simple closed curve $d_i\subset S_i$ $(i=1,2)$ 
which bounds with $c\cup d$ a pair of pants $P_i$. 
Write $\Sigma_i=S_i-P_i$; the genus of $\Sigma_i$ equals $g_i$.
Glue $P_1$ to $P_2$ along $c\cup d$ so that the resulting surface
$\Sigma_0$ is a two-holed torus containing $c\cup d$ in its interior.
 Choose a 
 nonseparating simple closed
 curve $e\subset \Sigma_0$ which intersects both $c,d$ in a single point.
Since $\phi(c)=\pm 1$ we have  
 $\phi(T_ce)=\phi(e)\pm 1$ where $T_c$ is 
 the left Dehn twist about $c$. Thus  
 via replacing $e$ by $T_c^ke$ for a suitable choice of 
 $k\in \mathbb{Z}$ we may assume 
 that $\phi(e)=1$. In other words, we may assume that 
 $e$ is a vertex of ${\cal C\cal G}_1$. 
 
Assume for the moment that $g_2\geq 1$.  
  By Lemma \ref{torus}, there exist simple closed curves $a\subset \Sigma_1,
  b\subset \Sigma_2$ with $\phi(a)=\phi(b)=0$. Connect $a$ to $b$ by an
  embedded arc $\epsilon$ which is disjoint from $c\cup e$ (and crosses through 
  the curve $d$). The curve $a+_\epsilon b$ satisfies $\phi(a+_\epsilon b)=1$, 
and it  is disjoint from both $c$ and $e$. Moreover, 
 the surfaces $S-(c\cup a+_\epsilon b)$ and $S-(e\cup a+_\epsilon b)$ 
  are connected.
 As a consequence, 
 $c$ can be connected to $e$ by an edge path in ${\cal C\cal G}_1^+$
  of length two which passes through $a+_\epsilon b$. 
  
  By symmetry of this construction, $e$ can also be connected to $d$ by an edge
  path in ${\cal C\cal G}_1^+$ and hence $c$ can be connected to $d$ by such a path.
  This completes the proof in the case that the genus $g_2$ of $S_2$ is positive. 
  
  If the genus of $S_2$ vanishes then the genus of $S_1$ equals $g_1=g-1\geq 2$. 
  Any nonseparating  simple closed curve in $S_1$ forms with both $c,d$ a nonseparating pair. 
  To find such a curve $e$ with $\phi(e)=1$, note that 
  $S_1$ contains two disjoint one-holed tori $T_1,T_2$, and 
  by Lemma \ref{torus}, there are embedded
  simple closed curves $a_i\in T_i$ which satisfy $\phi(a_i)=0$. Then   
  for any arc $\epsilon$ in $S_1$ connecting $a_1$ to $a_2$,
  the curve $e=a_1 +_\epsilon a_2$ is nonseparating, and it is connected with both
  $c,d$ by an edge in ${\cal C\cal G}_1^+$.  
  This is what we wanted to show.
 \end{proof}  

Proposition \ref{genus3}, Proposition \ref{connected6} and Lemma \ref{next} together show

\begin{corollary}\label{connected}
Let $\phi$ be a $\mathbb{Z}/r\mathbb{Z}$-spin structure on a closed surface
$\Sigma$ of genus $g\geq 3$. Then the graph
${\cal C\cal G}_1^+$ is connected.
\end{corollary}

  \section{The action of ${\rm Mod}(S)[\phi]$ on geometrically
    defined graphs}\label{theaction}

In this section we consider 
an arbitrary $\mathbb{Z}/r\mathbb{Z}$-spin structure $\phi$ on
a compact surface $S$ of genus $g\geq 3$, possibly with boundary,
 for some number $r\geq 2$.
Our goal is to gain some information on the 
stabilizer ${\rm Mod}(S)[\phi]$ of $\phi$ through its action
on the graph ${\cal C\cal G}_1^+$ introduced in Section \ref{graphsofcurves}.

We begin with 
some information on the stabilizer of a spin structure
$\phi$ on a compact surface $S$ with boundary. 
Fix a boundary component
$C$ of $S$.  Denote by $P_C{\rm Mod}(S)$ the subgroup of the 
mapping class group 
${\rm Mod}(S)$ of $S$ which fixes the boundary component $C$. 
Note that as we allow that a mapping class in $P_C{\rm Mod}(S)$ exchanges
boundary components of $S$ different from $C$, the group 
$P_C{\rm Mod}(S)$ coincides with the pure mapping 
class group of $S$ only if the boundary of
$S$ consists of one or two components.

Write $P_C{\rm Mod}(S)[\phi]$ to denote the stabilizer of $\phi$ in 
$P_C{\rm Mod}(S)$. This is a subgroup of $P_C{\rm Mod}(S)$ of finite index.
Let $\Sigma$ be the surface obtained from $S$ by
attaching a disk to $C$.
There is an embedding
$S\to \Sigma$ which induces a surjective homomorphism 
\[\Pi:P_C{\rm Mod}(S)\to {\rm Mod}(\Sigma).\] 

By a result of Johnson, extending earlier work of Birman
(see Section 4.2.5 of \cite{FM12}), there is an exact sequence
\begin{equation}\label{birman1}
1\to \mathbb{Z}\to {\rm ker}(\Pi)\xrightarrow{\Upsilon} \pi_1(\Sigma)\to 1\end{equation}
where $\mathbb{Z}$ is the infinite cyclic central subgroup of $P_C{\rm Mod}(S)$ 
generated by the Dehn twists about $C$ and where $\pi_1(\Sigma)$ is a
so-called point pushing group.  

For the formulation of the following lemma, recall that the integral
homology $H_1(\Sigma,\mathbb{Z})$ 
of a compact surface $\Sigma$ of genus $g\geq 2$,
possibly with boundary,
is a free abelian group
$\mathbb{Z}^{h}$ for some $h\geq 4$. In fact, $h=2g$ if the boundary of
$\Sigma$ is empty or connected, and in this case
this group is generated by the homology classes of 
nonseparating simple closed curves on $\Sigma$. If the boundary of
$\Sigma$ is disconnected, then it is still true that 
$H_1(\Sigma,\mathbb{Z})$ is generated by simple closed possibly
peripheral curves.

For $m\geq 1$ let $\Lambda_m\subset \pi_1(S)$ be the subgroup defined by the 
exact sequence
\[0\to \Lambda_m\to \pi_1(S)\to H_1(S,\mathbb{Z}/m\mathbb{Z})\to 0.\]
If  $\zeta:\pi_1(S)\to H_1(S,\mathbb{Z})$ denotes 
the natural surjective projection, then $\Lambda_m$ is 
the preimage under $\zeta$ of the lattice in 
$H_1(S,\mathbb{Z})$
generated by $m$ times the simple 
loop generators, and it is a subgroup
of $\pi_1(S)$ of finite index. 
Using the notations from the previous paragraph we have

\begin{lemma}\label{pointpush1}
Assume that the boundary circle $C$ is equipped with the orientation
induced from the orientation of $S$. 
\begin{enumerate}
\item If $\phi(C)=-1$ then 
$\Upsilon({\rm ker}\,\Pi\cap P_C{\rm Mod}(S)[\phi])= 
\pi_1(\Sigma)$.
\item 
If $\phi(C)=1$,  
then $\Upsilon({\rm ker}\, \Pi\cap P_C{\rm Mod}(S)[\phi])=\Lambda_{m}$
where $m=r/2$ if $r$ is even, and $m=r$ otherwise.
\end{enumerate}
\end{lemma}
\begin{proof}
Choose a basepoint $p$ for $\pi_1(\Sigma)$ in the interior of the
attached disk. 
Let $\alpha\subset \Sigma$ be a simple nonseparating loop through the basepoint $p$. 
Up to homotopy,  the oriented boundary
of a tubular neighborhood of $\alpha$ consists of two simple closed curves
$c_1,c_2$ which enclose the circle $C$. In other words, together with $C$ the
curves $c_1,c_2$ 
bound a pair of pants $P$ in $S$. We equip the curves $c_i$ with
the orientation as boundary curves of $P$. 

By Proposition 3.8 of \cite{Sa19}, we have
\begin{equation}\label{pairofpants}
\phi(C)+\phi(c_1)+\phi(c_2)=-1\end{equation}
and hence if $\phi(C)=-1$ then $\phi(c_1)+\phi(c_2)=0$.

Let as before $T_d$ be the left Dehn twist about a
simple closed curve $d$.
Let $\beta\subset S$ be an oriented simple closed curve which crosses through 
the pair of pants $P$.  As $c_1,c_2$ are disjoint, we have
$\iota(T_{c_2}^{-1}(\beta),c_1)=\iota(\beta,c_1)$ and therefore 
Definition \ref{spin} shows that
\begin{align}\label{intersectcom}
\phi(T_{c_1}T_{c_2}^{-1}(\beta)) &=\phi(T_{c_2}^{-1}(\beta))+
\iota(\beta,c_1)\phi(c_1)\\ &=
\phi(\beta)+\iota(\beta,c_1)\phi(c_1)-
\iota(\beta,c_2)\phi(c_2).\notag
\end{align}
On the other hand, as $c_1+c_2$ is homologous to the boundary curve $C$,
the homological intersection number fulfills 
$\iota(\beta,c_1+c_2)=0$. Hence from (\ref{pairofpants}) 
we conclude that if $\phi(C)=-1$ then 
$\phi(T_{c_1}T_{c_2}^{-1}(\beta))=
\phi(\beta)$. Since $\beta$ was an arbitrary simple closed curve, 
this shows that $T_{c_1}T_{c_2}^{-1}\in P_C{\rm Mod}(S)[\phi]$. 
But
$T_{c_1}T_{c_2}^{-1}\in P_C{\rm Mod}(S)$ is just the point-pushing map 
about $\alpha$ and therefore 
$\alpha$ is contained in 
$\Upsilon(P_C{\rm Mod}(S)[\phi])$. We refer to \cite{FM12} for a comprehensive
discussion of the various versions of the Birman exact sequence. 

As the point pushing group $\pi_1(\Sigma)$ is generated by point pushing maps
along simple loops, this shows the first part of the lemma.

To show the second part of the lemma, assume now that 
$\phi(C)=1$.  Equation (\ref{pairofpants}) shows that 
$\phi(c_1)+\phi(c_2)=-2$ and hence by Formula (\ref{intersectcom}) we have
\[\phi(T_{c_1}T_{c_2}^{-1}(\beta))=\phi(\beta) +\iota(\beta,c_1)\phi(c_1)+\iota(\beta,c_2)(\phi(c_1)+2).\]

Now let us assume that the oriented simple closed curve 
$\beta$ crosses a single time through $c_1$, say when 
it enters $P$. Then $\iota(\beta,c_1)=-1,\iota(\beta,c_2)=1$ and hence
\begin{equation}\label{add5}
\phi(T_{c_1}T_{c_2}^{-1}(\beta))=\phi(\beta)-\phi(c_1)+\phi(c_1)+2=
\phi(\beta)+2.\end{equation}  
Using this formula $r/2$ times if $r$ is even, and $r$ times if $r$ is odd, 
we conclude that
the point pushing map about $\alpha$ 
is not contained in ${\rm Mod}(S)[\phi]$, but it is the case for 
its $r/2$-th power or $r$-th power, respectively. 
Namely, putting $m=r/2$ if $r$ is even and $m=r$ otherwise, 
it follows from the above discussion that we have
$\phi((T_{c_1}T_{c_2}^{-1})^{m}(\beta))=\phi(\beta)$
for every simple
closed curve $\beta$ which either is disjoint from $P$ or which crosses
through $P$ precisely once. As such curves span the first homology of 
$S$, we conclude that the pull-back of $\phi$ under 
$(T_{c_1}T_{c_2}^{-1})^{m}$ coincides with $\phi$ on a
collection of simple closed
curves which span $H_1(S,\mathbb{Z})$. 
Corollary 2.6 of \cite{HJ89}  then shows that indeed, 
$(T_{c_1}T_{c_2}^{-1})^{m}\in P_C{\rm Mod}(S)[\phi]$. 
Moreover, by equation (\ref{add5}), we know that 
$(T_{c_1}T_{c_2}^{-1})^k\not\in P_C{\rm Mod}(S)[\phi]$ if 
$k$ is not a multiple of $m$. 

On the other hand, by Lemma 3.15 of \cite{Sa19}, 
Dehn twists about separating simple closed curves in 
$S$ are contained in ${\rm Mod}(S)[\phi]$. 
As the commutator subgroup of 
$\pi_1(\Sigma)$ is generated by simple closed separating curves, 
and for each such curve $\alpha$ both Dehn twists 
$T_{c_1},T_{c_2}$ about the boundary curves of a tubular neighborhood
of $\alpha$ as above are contained in $P_C{\rm Mod}(S)[\phi]$, 
this yields the
second part of the lemma. 
\end{proof}

Consider again an arbitrary compact surface $S$ 
of genus $g\geq 2$, equipped with a
$\mathbb{Z}/r\mathbb{Z}$-spin structure $\phi$ for some $r\geq 2$.
We use Lemma \ref{pointpush1} to 
analyze the action of ${\rm Mod}(S)[\phi]$ on the graph
${\cal C\cal G}_1^+$. We begin with the investigation of 
the stabilizer of a vertex $c$ of ${\cal C\cal G}_1^+$ in 
${\rm Mod}(S)[\phi]$. As ${\rm Mod}(S)[\phi]$ is a subgroup of 
${\rm Mod}(S)$ of finite index, the stabilizer 
${\rm Stab}(c)[\phi]$ of $c$ in ${\rm Mod}(S)[\phi]$ is a subgroup of 
finite index of the stabilizer ${\rm Stab}(c)$ of $c$ in ${\rm Mod}(S)$. 

The group ${\rm Stab}(c)$ can be described as follows. 
Cut $S$ open 
along $c$. The result is a surface $\Sigma^2$ of 
genus $g-1$ with two distinguished boundary components $C_1,C_2$.
These components are equipped with an orientation as subsets of the
oriented boundary of $\Sigma^2$. To simplify notations, let 
${\rm Mod}(\Sigma^2)$ be the subgroup of the mapping class group
of $\Sigma^2$ which preserves the subset $C_1\cup C_2$ of the boundary.
We allow that an element of ${\rm Mod}(\Sigma^2)$ 
exchanges $C_1$ and $C_2$. 
The stabilizer ${\rm Stab}(c)$ of $c$ in the mapping class group ${\rm Mod}(S)$ 
of $S$ can be identified with the quotient of the group ${\rm Mod}(\Sigma^2)$
by the relation $T_{C_1}T_{C_2}^{-1}=1$ where $T_{C_i}$ denotes
the left Dehn twist about the boundary circle $C_i$ (Theorem 3.18 of
\cite{FM12}). In short, we have 
\[{\rm Stab}(c)={\rm Mod}(\Sigma^2)/\mathbb{Z}.\] 

The infinite cyclic subgroup of ${\rm Stab}(c)$ generated by
the Dehn twist about $c$ is central. 
The quotient group ${\rm Stab}(c)/\mathbb{Z}$ can naturally be identified  
with the mapping class group ${\rm Mod}(\Sigma_2)$ of a surface
of genus $g-1$ with two punctures
and perhaps with boundary if the boundary of $S$ is non-trivial. 
We refer to \cite{FM12} for a comprehensive discussion of these facts. 

Let $\Sigma$ be the surface obtained from $\Sigma_2$ by forgetting the
punctures. Alternatively, $\Sigma$ is obtained from $\Sigma^2$ by 
attaching a disk to each boundary component. The group 
${\rm Mod}(\Sigma_2)={\rm Stab}(c)/\mathbb{Z}$ 
fits into the \emph{Birman exact sequence}
\begin{equation}\label{birman3}
1\to \pi_1(C(\Sigma,2))\xrightarrow{\rho} {\rm Stab}(c)/\mathbb{Z}\to
  {\rm Mod}(\Sigma)\to 1\end{equation}
  where $\pi_1(C(\Sigma,2))$ is the \emph{surface braid group}, that is, the fundamental
  group of the configuration space of two unordered distinct points in $\Sigma$. 
In particular, $\pi_1(C(\Sigma,2))$ is a normal subgroup of
${\rm Stab}(c)/\mathbb{Z}={\rm Mod}(\Sigma_2)$.

The surjective homomorphism
\[\theta:{\rm Stab}(c)\to {\rm Stab}(c)/\mathbb{Z}={\rm Mod}(\Sigma_2)\] restricts to 
a homomorphism ${\rm Stab}(c)[\phi]\to {\rm Mod}(\Sigma_2)$. The next proposition
gives some first information on its image under the assumption that 
$\phi$ is a $\mathbb{Z}/2\mathbb{Z}$-spin structure and 
$\phi(c)=1$.

\begin{proposition}\label{surject}
Let $\phi$ be a $\mathbb{Z}/2\mathbb{Z}$-spin structure on $S$ and let
$c$ be a simple closed curve with $\phi(c)=1$. Then 
  $\rho(\pi_1(C(\Sigma,2)))\subset \theta({\rm Stab}(c)[\phi])$.  
\end{proposition}
\begin{proof} Let $\pi_1(PC(\Sigma,2))$ be the intersection of the 
fibre of the Birman exact sequence (\ref{birman3})
with the subgroup of ${\rm Mod}(\Sigma_2)$
which fixes each of the two distinguished punctures. 
Following Section 4.2.5 of \cite{FM12},
the group $\pi_1(PC(\Sigma,2))$ can be described as follows. 

Let $C_1,C_2$ be the distinguished boundary components of the surface 
$\Sigma^2=S-c$. Let $\Sigma^1$ be the surface 
obtained from $\Sigma^2$ by attaching a disk to the boundary circle 
$C_1$. 
Let $P{\rm Stab}(c)$ and $P{\rm Mod}(\Sigma^2)$
be the index two subgroup of ${\rm Stab}(c)$ and ${\rm Mod}(\Sigma^2)$ 
which preserves each of the
two boundary components $C_1,C_2$ 
of $S-c$.
The inclusion $\Sigma^2\to \Sigma^1$ induces a surjective homomorphism
\[\Xi:P{\rm Stab}(c)/\mathbb{Z}\to {\rm Mod}(\Sigma^1)/\mathbb{Z}\] where as before
${\rm Mod}(\Sigma^1)$ is required to fix 
the boundary component $C_2$ of $\Sigma^1$
and where the group $\mathbb{Z}$ acts as the
group of Dehn twists
about $c$ and about $C_2$. 
The kernel ${\rm ker}(\Xi)$ 
of this homomorphism is isomorphic to $\pi_1(\Sigma^1)$ (see \cite{FM12} for more
information on this version of the Birman exact sequence).

The spin structure $\phi$ pulls back to a spin structure
$\hat \phi$ on $\Sigma^2$. 
Since $\phi$ is a $\mathbb{Z}/2\mathbb{Z}$-spin structure on 
$S$ and $\phi(c)=1$, the value of 
$\hat \phi$ 
on each of the two boundary circles $C_1,C_2$
coincides with the value of 
a spin structure on the boundary of an embedded disk. This implies that 
$\hat \phi$ induces a spin structure $\phi^\prime$ on $\Sigma^1$.
Or, equivalently, $\hat \phi$ is the pull-back of a spin structure 
$\phi^\prime$ on $\Sigma^1$ via the inclusion $\Sigma^2\to \Sigma^1$. 
 By Lemma \ref{pointpush1}, the group ${\rm ker}(\Xi)=
\pi_1(\Sigma^1)$ stabilizes $\hat \phi$, that is, we have 
${\rm ker}(\Xi)\subset {\rm Mod}(\Sigma^2)[\hat \phi]$. 

Apply Lemma \ref{pointpush1} a second time 
to the homomorphism
${\rm Mod}(\Sigma^1)/\mathbb{Z}\to {\rm Mod}(\Sigma)$ where $\Sigma$ is obtained from 
$\Sigma^1$ by attaching a disk to $C_2$. 
As the group $\pi_1(PC(\Sigma,2))$ can be 
described as the quotient by its center $\mathbb{Z}^2$ 
of the kernel of the homomorphism $P{\rm Mod}(\Sigma^2)\to {\rm Mod}(\Sigma)$ 
which is obtained 
by applying the Birman
exact sequence twice, first to a map which caps off the boundary component
$C_1$, followed
by the map which caps off $C_2$, this shows that $\pi_1(PC(\Sigma,2))\subset
\theta ({\rm Stab}(c)[\phi]$. As exchanging $C_1$ and $C_2$ also preserves
$\hat \phi$ the proposition follows. 
\end{proof}

We are now ready to give a complete description of the stabilizer 
in ${\rm Mod}(S)[\phi]$ of 
a nonseparating simple closed curve $c$ on $S$ with $\phi(c)=1$
where as before, $\phi$ is 
a $\mathbb{Z}/2\mathbb{Z}$-spin
structure on a compact surface $S$ of genus $g\geq 3$.

Cut $S$ open along $c$ and write
$\Sigma^2=S-c$. The spin structure $\phi$ of $S$ pulls back to a 
$\mathbb{Z}/2\mathbb{Z}$-spin structure 
$\hat \phi$ on $\Sigma^2$. 
Denote as before by $\Sigma$ the surface of genus $g-1$ with empty or connected
boundary obtained from
$\Sigma^2$ by capping off the two distinguished boundary
components. We have

\begin{proposition}\label{parity}
  The $\mathbb{Z}/2\mathbb{Z}$-spin structure $\phi$ on $S$ induces
  a $\mathbb{Z}/2\mathbb{Z}$-spin structure $\phi_c$ on
  $\Sigma$ whose parity 
   coincides with the parity of $\phi$. If $\Pi:{\rm Stab}(c)/\mathbb{Z}\to 
   {\rm Mod}(\Sigma)$ denotes the surjective homomorphism induced by the 
  inclusion $S-c\to \Sigma$ then 
  \[\Pi^{-1}{\rm Mod}(\Sigma)[\phi_c]={\rm Stab}(c)[\phi]/\mathbb{Z}.\] 
\end{proposition}
\begin{proof} As $\phi$ is a $\mathbb{Z}/2\mathbb{Z}$-spin structure,
the value of $\phi$ on a boundary circle of $S-c$ corresponding to a copy of $c$
coincides with the value of a $\mathbb{Z}/2\mathbb{Z}$-spin 
structure on the boundary of a disk. Thus 
$\phi$ induces a spin structure $\phi_c$ on $\Sigma$.

To compare the parities of 
the spin structures $\phi$ and $\phi_c$,
  assume that
  $\Sigma$ is obtained from $S-c$ by attaching disks $D_1,D_2$
  to the two boundary components of $S$ which correspond to the
  two copies of $c$. 
  Choose a geometric symplectic basis $a_1,b_1,\dots,a_{g-1},b_{g-1}$ 
  for $\Sigma$, consisting of simple closed
  oriented curves which do not intersect the disks $D_1,D_2$.
  Then
  $a_1,b_1,\dots,a_{g-1},b_{g-1}$ can be viewed as a system of curves
  in $\Sigma^2=\Sigma-(D_1\cup D_2)$ which maps to a curve system with the same
  properties in $S$ by the map $\Sigma^2\to S$. This curve system  
  can be extended to a geometric symplectic
  basis for $S$
  containing the 
  curve $c$, equipped with any orientation. As $\phi(c)=1$
  we have
  $\phi(c)+1=0$. The claim now follows from
  the fact that $\phi_c(u)=\phi(\hat u)$ for
  $u\in \{a_1,b_1,\dots,a_{g-1},b_{g-1}\}$ where $\hat u$ is the image of $u$
  under the inclusion $\Sigma^2\to S$, together with 
  the formula (\ref{arf}) for the
  Arf invariant. 

We are left with showing that ${\rm Stab}(c)[\phi]/\mathbb{Z}=
\Pi^{-1}{\rm Mod}(\Sigma)[\phi_c]$. Observe first that as $\phi_c$ is induced
from $\phi$, we have $\Pi{\rm Stab}(c)[\phi]/\mathbb{Z}\subset {\rm Mod}(\Sigma)[\phi_c]$. 

To show that in fact equality holds let $\Sigma_2$ be the 
surface obtained from $S-c$ by replacing the boundary components by punctures. 
The group ${\rm Stab}(c)[\phi]/\mathbb{Z}$ can be identified with 
a subgroup $\Gamma_c$ of ${\rm Mod}(\Sigma_2)$.
We view the punctures of $\Sigma_2$ as marked points $p_1,p_2$ in $\Sigma$.

Let $\theta$ be any diffeomorphism of $\Sigma$ which preserves $\phi_c$. 
Then $\theta$ is isotopic to a diffeomorphism of $\Sigma$ which  equals the
identity on a disk $D\subset \Sigma$
containing both points $p_1,p_2$. Thus $\theta$ lifts to a
diffeomorphism $\theta^\prime$ of $\Sigma_2$ which preserves the pull-back of 
$\phi_c$ to a spin structure on $\Sigma_2$.

The  
boundary circle $\partial D$ of $D$ can be 
viewed as a simple closed curve in $S-c$. Via the projection
$S-c\to S$ which identifies the two distinguished boundary components of $S-c$,
the curve $\partial D$ 
 projects to a separating simple closed curve in $S$
which decomposes $S$ into a one-holed torus $T$ containing $c$ and a surface of 
genus $g-1$ with connected boundary. The diffeomorphism 
$\theta^\prime$ lifts to a diffeomorphism
$\Theta$ of $S$ which is the identity on $T$.

Then $\Theta^*\phi$ is a spin structure on $S$ which defines the same
function on $H_1(S,\mathbb{Z})$ as $\phi$.
Using once more the result of Humphries and Johnson \cite{HJ89}
(see Theorem 3.9 of \cite{Sa19}), this implies that
$\Theta$ stabilizes $\phi$. As $\Theta$ projects to the mapping class 
of $\Sigma$ defined by
the diffeomorphism $\theta$, this 
shows surjetivity of the 
homomorphism $\Pi:{\rm Stab}(c)[\phi]/\mathbb{Z}\to {\rm Mod}(\Sigma)[\phi_c]$. 

On the other hand, by Proposition \ref{surject} the kernel of 
the homomorphism $\Pi$ also is contained
in ${\rm Stab}(c)[\phi]/\mathbb{Z}$. Together this completes the proof of the proposition.
\end{proof}

The next observation uses 
Proposition 4.9
of \cite{Sa19}. For its formulation, recall from Section \ref{graphsofcurves}
the definition of the graph ${\cal C\cal G}_1^+$. Its vertices
are nonseparating simple closed 
curves with prescribed value $\pm 1$ of the spin structure.  
The graph ${\cal C\cal G_1}^+$ is well defined if the genus $g$ of $S$ is at least two
although it may not have edges.

Note that in the statement of Proposition \ref{connect}, we allow that the 
surface $S$ has non-empty boundary, and we consider 
$\mathbb{Z}/r\mathbb{Z}$-spin structures where $r$ may be larger than $2g-2$.
This is crucial for  
an inductive approach towards higher spin structures via cutting surfaces
open along separating simple closed curves, and it is used in the proof of 
Theorem \ref{main3}. 


\begin{proposition}\label{connect}
Let $\phi$ be a $\mathbb{Z}/r\mathbb{Z}$-spin structure on a compact surface 
$S$ of genus $g\geq 2$ with empty or connected boundary. 
Then for any two directed edges 
$e_1,e_2$ of the graph ${\cal C\cal G}_1^+$ 
there exists a mapping class 
$\zeta\in {\rm Mod}(S)[\phi]$ with $\zeta(e_1)=e_2$.
In particular, the action of
  ${\rm Mod}(S)[\phi]$ on  ${\cal C\cal G}_1^+$ is vertex transitive.
\end{proposition}  
\begin{proof} The proof consists of an adjustment of the argument in the proof 
of Proposition 4.9 of \cite{Sa19}. 

Recall that 
a geometric symplectic basis for $S$
is a set $\{a_1,b_1,\dots,a_{2g},b_{2g}\}$ of simple closed curves
on $S$ such that $a_i,b_i$ intersect in a single point, and 
$a_i\cup b_i$ is disjoint from $a_j\cup b_j$ for $j\not=i$. 

A vertex of ${\cal C\cal G}_1^+$ is a simple closed curve $c$ on 
$S$ with $\phi(c)=\pm 1$. In the sequel we always orient such a vertex $c$ 
in such a way that $\phi(c)=1$. 
For a given directed edge $e$ of 
${\cal C\cal G}_1^+$ with ordered endpoints $c,d$, 
we aim at constructing 
a geometric symplectic basis ${\cal B}(e)$ such that
$a_1=c,a_2=d,
\phi(a_i)=0$ for $i\geq 3$, $\phi(b_i)=0$ for $i\leq g-1$ and
$\phi(b_g)=0$ or $1$ as predicted by the parity of $\phi$.
If such a basis ${\cal B}(e_1),{\cal B}(e_2)$ 
can be found for any two directed edges $e_1,e_2$ of ${\cal C\cal G}_1^+$
with ordered endpoints $c_1,d_1$ and $c_2,d_2$, then
there exists a diffeomorphism $\zeta$ of $S$ which maps ${\cal B}(e_1)$ to 
${\cal B}(e_2)$ and maps $c_1,d_1$ to $c_2,d_2$. 
The pullback $\zeta^*\phi$ of $\phi$ is a spin structure on $S$ whose values
on ${\cal B}(e_1)$ coincide with the values of $\phi$. By a result of 
Humphries and Johnson \cite{HJ89}, see Theorem 3.9 of \cite{Sa19},
this implies that $\zeta^*\phi=\phi$ and hence the isotopy class of
$\zeta$ is contained in ${\rm Mod}(S)[\phi]$ and maps 
the directed edge $e_1$ to the directed edge $e_2$.

To simplify further, 
choose any geometric symplectic
basis 
\[{\cal B}=\{\alpha_1,\beta_1,\dots,\alpha_g,\beta_g\}\] for $S$ 
with $\alpha_1=c$, $\alpha_2=d$. A small tubular neighborhood of 
$\alpha_i\cup \beta_i$ is a one-holed torus $T_i$ embedded in $S$. 
By Lemma \ref{torus}, for all $i\geq 3$ we may replace $\alpha_i$ by an oriented 
simple closed curve in $T_i$, again denoted by $\alpha_i$, which satisfies
$\phi(\alpha_i)=0$. 

Assume that $\beta_i$ $(i=1,2)$ 
is oriented in such a way that 
$\iota(\beta_i,\alpha_i)=1$ where
$\iota$ is the symplectic form. As $\phi(T_{\alpha_i}(\beta_i))=
\phi(\beta_i)+1$, 
via perhaps replacing $\beta_i$ by its image under a suitably chosen power of
a Dehn twist about $\alpha_i$ we may assume that $\phi(\beta_i)=0$.
Therefore for the construction of a geometric symplectic
basis ${\cal B}(e)$ with the required properties, it suffices 
to modify successively the curves $\beta_i$ 
$(i\geq 3)$ while keeping $\alpha_j$ $(j\geq 1)$ and $\beta_k$ for 
$k<i$ fixed
such that $\phi$ assumes the prescribed values on the
modified curves.

We follow the proof of Proposition 4.9 of \cite{Sa19}. For $1\leq i\leq g$ 
let $\delta_i$ be the boundary curve
of the torus $T_i$ which is a small tubular neighborhood of $\alpha_i\cup \beta_i$,  
equipped with the orientation as
an oriented boundary circle of $S-T_i$ $(i\geq 1)$. 
By homological coherence (Proposition 3.8 of \cite{Sa19}), 
we have $\phi(\delta_i)=1$ for all $i$.

Thus if $\epsilon$ is an embedded arc in $S$
connecting $\beta_3$ to
$\delta_1$ whose interior is 
disjoint from $\alpha_3$ and all $\delta_j$
for $j\not=3$, then  
$\phi(\beta_3+_\epsilon \delta_1)=\phi(\beta_3)+2$.
Moreover, $\beta_3+_\epsilon \delta_1$ is disjoint from 
$\delta_j$ for all $j\not= 3$.

A small tubular neighborhood of $\alpha_3\cup (\beta_3+_\epsilon \delta_1)$ is 
a one-holed torus $\hat T_3$ disjoint from the tori $T_i$ for $i\not=3$. 
Thus we can repeat this construction
with an arc connecting $\beta_3+_\epsilon \delta_1$ to $\delta_1$
which is disjoint from $\alpha_3$ and whose interior is 
disjoint from all $\delta_j$ for $j\not=3$. Repeating further if necessary,  
we can find a simple closed curve $\beta_3^\prime$ intersecting
$\alpha_3$ in a single point and disjoint from the curves
$\delta_j$ for $j\not= 3$ so that $\phi(\beta_3^\prime)\in \{0,1\}$.

Let $\delta_3^\prime$ be the boundary of a tubular neighborhood of
$\alpha_3\cup \beta_3^\prime$. Then $\delta_3^\prime$ is disjoint
from all the curves $\delta_j$ for $j\not= 3$. 
As in the proof of Proposition 4.9 of \cite{Sa19}, repeat this
procedure with the curve $\beta_4$ and the curves
$\delta_1,\delta_2,\delta_3^\prime,\dots,\delta_{g}$. In finitely many
steps 
we can change the geometric 
symplectic basis ${\cal B}$
to a geometric symplectic basis
${\cal B}^\prime=\{\alpha_1,\beta_1,\alpha_2,\beta_2,\alpha_3,
\beta_3^{\prime},\dots,\alpha_{g},\beta_{g}^{\prime}\}$
which fulfills $\phi(\beta_j^{\prime})=0$ or $1$ for all $3\leq j\leq g$.

It remains to further alter $\beta_j^\prime$ for $3\leq j\leq g-1$
to a nonseparating simple closed curve $\beta_j^{\prime\prime}$ with 
$\phi(\beta_j^{\prime\prime})=0$, and to alter
$\beta_g^\prime$ to a simple closed curve $\beta_g^{\prime\prime}$
with  
$\phi(\beta_{g}^{\prime\prime})=0$ or $1$ depending on the parity of
the $\mathbb{Z}/r\mathbb{Z}$-spin structure $\phi$. 
This construction is carried out in detail in the proof of Proposition 4.9 of \cite{Sa19}
and will not be presented here as it would require the introduction of a significant amount
of new notation. It takes place in a subsurface of $S$
of genus $g-2$ which is disjoint from
$\alpha_1,\beta_1,\alpha_2,\beta_2$ 
and contains $\alpha_i,\beta_i$ for $3\leq i\leq g$. 
The resulting geometric
symplectic basis has the properties we are looking for.
\end{proof}

\begin{remark}
The proof of Proposition \ref{connect} can also be used to show the
following. Under the assumption of the proposition, let $c,d\subset S$ 
be two nonseparating simple closed curves with $\phi(c)=\phi(d)=0$; then
there exists some $\zeta\in {\rm Mod}(S)[\phi]$ with $\zeta(c)=d$. In fact,
this case is more explicitly covered by Proposition 4.2 and Proposition 4.9 of 
\cite{Sa19}. 
\end{remark}

The next statement is an extension of Proposition \ref{connect} to surfaces
with more than one boundary component under some restrictions on the 
spin structure.

\begin{corollary}\label{twoconnect}
Let $\phi$ be a $\mathbb{Z}/r\mathbb{Z}$-spin structure on a compact surface 
$S$ of genus $g\geq 2$ 
with non-empty boundary which is induced from a spin structure $\phi^\prime$ on
a compact surface $\Sigma$ of genus $g$ with empty or connected boundary
by an inclusion $S\to \Sigma$
which maps each boundary component of $S$ to the boundary of an
embedded disk in $\Sigma$. 
Then for any two vertices $c,d$ of ${\cal C\cal G}_1^+$ there exists a mapping class
$\zeta\in {\rm Mod}(S)[\phi]$ with $\zeta(c)=d$. In particular, the action of 
${\rm Mod}(S)[\phi]$ is transitive on the vertices of ${\cal C\cal G}_1^+$. 
\end{corollary}
\begin{proof}
Let $\Psi:S\to \Sigma$ be the natural embedding. Let $c,d$ be vertices of the
graph ${\cal C\cal G}_1^+$ for the spin structure $\phi$ on $S$. 
Then $c,d$ are nonseparating simple closed curves and hence their images
$\Psi(c),\Psi(d)$ are nonseparating simple closed curves on $\Sigma$.
Furthermore, as $\phi$ is the pull-back of a spin structure $\phi^\prime$
on $\Sigma$, we have $\phi^\prime(\Psi(c))=
\phi^\prime(\Psi(d))=1$.

By Proposition \ref{connect}, 
there exists a mapping class
$\theta\in {\rm Mod}(\Sigma)(\phi^\prime)$ which maps $\Psi(c)$
to $\Psi(d)$. We can choose a diffeomorphism of $\Sigma$
representing $\theta$ which equals the identity on each component of
$\Sigma-S$. Thus there exists a lift 
$\Theta$ of $\theta$ to a mapping class of $S$. This mapping class 
is contained in ${\rm Mod}(S)[\phi]$, and it maps the
simple closed curve 
$c$ to a simple closed curve $d^\prime$ whose image under
 $\Psi$ is isotopic to $\Psi(d)$. 

Using once more the Birman exact sequence, this implies that there exists
a mapping class $\beta$ in the kernel of the homomorphism 
${\rm Mod}(S)\to {\rm Mod}(\Sigma)$  
which maps $d^\prime$ to $d$.
But by an iterated application of Lemma \ref{pointpush1}, 
this kernel is contained in ${\rm Mod}(S)[\phi]$
and hence $c$ can be mapped to $d$
by an element of ${\rm Mod}(S)[\phi]$. 
\end{proof}

The \emph{augmented Teichm\"uller space} $\overline{\cal T}(S)$ of 
the compact surface $S$ 
is the union of the
Teichm\"uller space with so-called \emph{boundary strata}. Each of these
boundary strata is defined by a non-empty system ${\cal C}$ of pairwise disjoint essential
simple closed curves. The stratum defined by such a curve system
can be thought of as the Teichm\"uller space of the surface obtained from 
$S$ by shrinking each component of ${\cal C}$ to a node. 
In other words, such a stratum is 
a complex manifold which is naturally biholomorphic to 
the Teichm\"uller
space of the surface obtained by cutting $S$ open along the
components of ${\cal C}$ and replacing each boundary component of the resulting
bordered surface by a puncture. 

Using Fenchel Nielsen coordinates, 
the augmented Teichm\"uller space can be equipped with
a natural topology. 
For this topology, the usual Teichm\"uller space embeds into $\overline{\cal T}(S)$
as an open dense subset. Furthermore, the inclusion of the Teichm\"uller space
of a punctured surface defined by the curve system
${\cal C}$ onto a boundary stratum of $\overline{\cal T}(S)$ also
is an embedding. We refer to \cite{Wol10} for an detailed description and for 
a discussion of the following

\begin{theorem}\label{augmented}
The augmented Teichm\"uller space $\overline{\cal T}(S)$ 
is a non locally compact stratified space. The mapping class group ${\rm Mod}(S)$ of 
$S$ acts on $\overline{\cal T}(S)$, with quotient the Deligne Mumford compactification of 
the moduli space of curves of genus $g$. 
\end{theorem}

Fix again a  $\mathbb{Z}/2\mathbb{Z}$-spin
structure $\phi$ on a surface $S$ of genus $g\geq 2$.
Define the 
\emph{spin Teichm\"uller space} ${\cal T}_{\rm spin}(S)$ 
to be the Teichm\"uller space 
of $S$ together with this spin structure. The group
${\rm Mod}(S)[\phi]$ acts on
${\cal T}_{\rm spin}(S)$ as a group of biholomorphic transformations,
with quotient the \emph{spin moduli space} 
${\cal M}_\phi={\cal T}(S)/{\rm Mod}(S)[\phi]$.

We can define an augmented spin 
Teichm\"uller space $\overline{\cal T}_{\rm spin}(S)$
 as the union of spin Teichm\"uller space  
 with all strata of augmented Teichm\"uller space which are defined by 
systems of 
nonseparating simple closed curves $c$ on $S$ with $\phi(c)=1$. 
Equipped with the subspace topology, 
this is a subspace of $\overline{\cal T}(S)$ which is invariant under
the action of the spin mapping class group. As a corollary of the discussion in 
this section, we have

\begin{corollary}\label{partialbord}
The quotient $\overline{\cal T}_{\rm spin}(S)/{\rm Mod}(S)[\phi]$ 
is a partial bordification of the spin moduli space ${\cal T}_{\rm spin}(S)/{\rm Mod}(S)[\phi]$. 
Its boundary contains the spin moduli
space of the same parity on a surface of genus $g-1$ with two marked points (punctures) 
as an open dense subset.
\end{corollary}

\begin{remark}
Corollary \ref{partialbord} can be thought of as describing a specific subset
of a Deligne Mumford compactification of the moduli space of curves with
a fixed spin structure. Such a Deligne Mumford compactification was constructed
by Cornalba \cite{Co89}. 
\end{remark}

\section{Structure of the spin mapping class group of odd parity}
\label{structureof}

The goal of this section is to prove the first part of Theorem \ref{main2}.

We begin with some additional information on the spin mapping class
 group. 
Fix a 
$\mathbb{Z}/r\mathbb{Z}$-spin structure $\phi$ 
on a closed surface $\Sigma_g$ of genus $g$ for some $r\geq 2$.
For a simple closed curve $c$ on $\Sigma_g$ with 
$\phi(c)=\pm 1$, this spin structure restricts to a spin structure
on the surface 
$\Sigma_{g-1}^2$ 
of genus $g-1$ with two boundary circles $c_1,c_2$ 
obtained by cutting $\Sigma_g$ open along $c$. We denote this spin
structure again by $\phi$. Define the 
group $\Gamma_{g-1}^2$ to be the following
quotient of the spin mapping class group 
${\rm Mod}(\Sigma_{g-1}^2)[\phi]$. 

The group ${\rm Mod}(\Sigma_{g-1}^2)[\phi]$ contains 
a rank two free abelian central subgroup generated by the
$r$-th powers of the left Dehn twists $T_{c_1},T_{c_2}$ 
about the 
boundary circles $c_1,c_2$ of $\Sigma_{g-1}^2$.  Define 
$\Gamma_{g-1}^2={\rm Mod}(\Sigma_{g-1}^2)[\phi]/\mathbb{Z}$
where the infinite cyclic subgroup $\mathbb{Z}$ is generated
by $T_{c_1}^rT_{c_2}^{-r}$. Then 
$\Gamma_{g-1}^2$ is isomorphic to the stabilizer 
in ${\rm Mod}(\Sigma_g)[\phi]$ of the curve $c$. 
Note that up to isomorphism, 
the group $\Gamma_{g-1}^2$ 
does not depend on the vertex $c\in {\cal C\cal G}_1$. Namely,
by  Proposition \ref{connect}, 
the stabilizers 
in ${\rm Mod}(\Sigma_g)[\phi]$
of nonseparating simple closed
curves $c$ with $\phi(c)=\pm 1$  
are all conjugate and hence isomorphic. 
 
Observe that the group $\Gamma_{g-1}^2$ is an infinite cyclic central extension of 
a finite index subgroup of the mapping class
group 
of a surface $\Sigma_{g-1,2}$ 
of genus $g-1$ with two punctures. 
Thus it makes sense
to talk about its action on isotopy classes of 
essential curves on the surfaces 
$\Sigma_{g-1,2}$ and $\Sigma_{g-1}^2$. The map $\Sigma_{g-1}^2\to 
\Sigma_{g-1,2}$ which contracts each boundary component to a puncture
defines a bijection on such isotopy classes.  

We have

\begin{proposition}\label{spinmap}
Let $\phi$ be a $\mathbb{Z}/r\mathbb{Z}$-spin structure on 
a closed surface $\Sigma_g$ of genus $g\geq 3$.  
There is a commutative diagram\\
\begin{equation}
\begin{tikzcd}
\Gamma_{g-1}^2 \arrow[r,"\iota_1"] \arrow[dr, "\iota_2"]
& \Gamma_{g-1}^2*_A\Gamma_{g-1}^2\rtimes \mathbb{Z}/2\mathbb{Z}
\arrow[d, "\rho"] \\
& {\rm Mod}(\Sigma_g)[\phi]
\end{tikzcd} 
\end{equation}
where the homomorphisms $\iota_1,\iota_2$ are inclusions, and the 
homomorphism $\rho$ is surjective. 
The subgroup $A$ of $\Gamma_{g-1}^2$ 
is the stabilizer in $\Gamma_{g-1}^2$ 
of a nonseparating simple closed curve $d$ 
on $\Sigma_{g-1}^2$ with $\phi(d)=\pm 1$. The
group $\mathbb{Z}/2\mathbb{Z}$ acts on $\Gamma_{g-1}^2*_A\Gamma_{g-1}^2$
by exchanging the two factors, and it acts as an automorphism
on $A$.  
\end{proposition}
\begin{proof}
Fix a pair of nonseparating simple closed disjoint curves 
$c,d$ on $\Sigma_g$ with $\phi(c)=\phi(d)=\pm 1$ which are connected
by an edge in the graph ${\cal C\cal G}_1^+$, that is, so that
$\Sigma_g-(c\cup d)$ is connected. 
Let $\Gamma_c,\Gamma_d\subset {\rm Mod}(\Sigma_g)[\phi]$ be the stabilizers of 
$c,d$ in the spin mapping class group of $\Sigma_g$. By Corollary \ref{twoconnect}, 
these groups 
are naturally isomorphic to the group $\Gamma_{g-1}^2$, and they
intersect in the index two subgroup $A=\Gamma_c\cap \Gamma_d$ 
of the stabilizer
of $c\cup d$ in 
${\rm Mod}(\Sigma_g)[\phi]$ consisting of all elements which preserve both $c,d$
individually. 
The full stabilizer of $c\cup d$ in ${\rm Mod}(\Sigma_g)[\phi]$ 
is a $\mathbb{Z}/2\mathbb{Z}$
extension of $\Gamma_c\cap \Gamma_d$, where the generator 
$\Phi$ of $\mathbb{Z}/2\mathbb{Z}$ acts as involution on 
$A=\Gamma_c\cap \Gamma_d$ exchanging $c$ and $d$. This involution 
extends to an involution of $\Gamma_c*_A\Gamma_d$ 
exchanging the two subgroups $\Gamma_c,\Gamma_d$.

By the universal property of free amalgamated products, 
there is 
a homomorphism 
\[\rho:\Gamma=\Gamma_c*_A\Gamma_d\rtimes \mathbb{Z}/2\mathbb{Z}
\to {\rm Mod}(\Sigma_g)[\phi]. \]
All we need to show is that $\rho$ is surjective, that is, that 
$\rho(\Gamma)= 
{\rm Mod}(\Sigma_g)[\phi]$.

As ${\rm Mod}(\Sigma_g)[\phi]$ acts transitively on the vertices of the graph 
${\cal C\cal G}_1^+$, for this it suffices to show that its subgroup 
$\rho(\Gamma)$ acts transitively on the vertices of ${\cal C\cal G}_1^+$
as well. Namely, by construction, the stabilizer of the vertex
$c$ of ${\cal C\cal G}_1^+$ in $\rho(\Gamma)$ coincides with its
stabilizer in ${\rm Mod}(\Sigma_g)[\phi]$. As $\rho(\Gamma)$ is a subgroup
of ${\rm Mod}(\Sigma_g)[\phi]$, this then implies equality.

To show transitivity of the action of $\rho(\Gamma)$ on the vertices
of ${\cal C\cal G}_1^+$ let $v\in {\cal C\cal G}_1^+$ be any vertex.
By Corollary \ref{connected}, 
the graph ${\cal C\cal G}_1^+$ is connected and hence we can find an edge path
$(c_i)\subset {\cal C\cal G}_1^+$ connecting $c_0=c$ to $c_k=v$. We also may assume that
$c_1=d$.

By the assumption $\phi(d)=\pm 1$, 
for one of the two boundary components $d_1,d_2$ of $\Sigma_g-d$, equipped with 
the orientation as a boundary component of $\Sigma_g-d$, say the component 
$d_1$, we have $\phi(d_1)=-1$. Thus we can attach a disk $D$ to $c_1$ and obtain
a surface $\Sigma^\prime$ with spin structure $\phi^\prime$ which induces 
the spin structure $\phi$ on $\Sigma_g-d$. As a consequence, 
the restriction of $\phi$
to $\Sigma_g-d$ fulfills the hypothesis in Corollary \ref{twoconnect}. 
As $c=c_0$ and $c_2$ are nonseparating 
simple closed curves in $\Sigma_g-d$ with 
$\phi(c)=\phi(c_2)=\pm 1$, 
Corollary \ref{twoconnect} shows that 
there exists an element 
$\Psi_1\in \Gamma_d\subset \rho(\Gamma)$ such that $\Psi_1(c)=c_2$. Then the stabilizer of 
$c_2$ in ${\rm Mod}(\Sigma_g)[\phi]$ equals $\Psi_1\Gamma_c\Psi_1^{-1}$ and hence
it is contained in $\rho(\Gamma)$. 
Thus we can apply 
Corollary \ref{twoconnect} to $\Psi_1\Gamma_c\Psi_1^{-1}$ and find an element
$\Psi_2\in \rho(\Gamma)$ which maps $c_1$ to $c_3$. 
Proceeding inductively and using the fact
that $\Gamma_c$ is conjugate to $\Gamma_d$ in $\rho(\Gamma)$ 
by the generator of the subgroup $\mathbb{Z}/2\mathbb{Z}$,  
this completes the proof of the proposition.
\end{proof}

Recall from the introduction the definition of
an admissible curve system on a closed surface $\Sigma_g$ of genus
$g\geq 2$.
The mapping class group of $\Sigma_g$ naturally acts on the family of all 
admissible curve systems on $\Sigma_g$. 
Recall also that the curve diagram of an admissible curve system is
a finite tree.


Since the curve diagram of an admissible curve system ${\cal C}$ 
is connected, 
each curve $c\in {\cal C}$ 
intersects at least one other simple closed curve on $\Sigma_g$ transversely
in a single point and hence it is nonseparating.

We need some technical information on admissible curve
systems. To this end let ${\cal C}$ be any admissible
curve system on an oriented surface $S$.
We require that the boundary of $S$ is empty, but we
allow for the moment that $S$ has punctures. For admissibility, 
we require
that all complementary components of ${\cal C}$ are either
topological disks or once punctured topological disks.

The union $\cup\{c\mid c\in {\cal C}\}$ is an
embedded graph $G$ in $S$ whose vertices are the intersection
points between the curves from ${\cal C}$.
Choose a basepoint
$x\in G$ which is contained in the interior of an edge of $G$.
This edge is contained in a simple closed curve $c_0\in {\cal C}$ which defines
a distinguished vertex $v_0$ 
in the curve diagram of ${\cal C}$.

Construct inductively a family $L$ of homotopy classes of 
loops in $G$ based at $x$ as follows.
Let $L_0$ be the family consisting of the two 
based loop which go once around
the simple closed curve
$c_0\in {\cal C}$ containing $x$ in either direction. 
Assume by induction that for some $k\geq 1$ 
we defined
a system of based loops $L_{k-1}$. 
Let $\{c_{k_1},\dots,c_{k_s}\}\subset {\cal C}$ 
be the curves in ${\cal C}$
whose distance in the curve diagram to the distinguished vertex $v_0$
equals $k$. Define
\[L_k=\{T_{c_{k_u}}^{\pm 1}d\mid u\leq s,d\in L_{k-1}\}\]
and let $L=L_b$ where $b\geq 1$ is the maximal distance of a vertex
in the curve diagram of ${\cal C}$ 
to the distinguished vertex $v_0$.

The following appears implicitly in \cite{PV96} and explicitly as 
Lemma 9.3 of \cite{Sa19}. 

\begin{lemma}\label{generate}
  The loops from the system $L$ generate the fundamental group 
  $\pi_1(S,x)$ of $S$.
\end{lemma}

As a consequence we obtain (see Lemma 9.4 of \cite{Sa19}). 

\begin{lemma}\label{generatepointpush}
  Let ${\cal C}$ be an admissible curve system on a surface
  $S$, possibly with punctures. Let $p$ be a puncture of $S$ and
  assume that there are two curves $c_1,c_2\in {\cal C}$ which bound
  a once punctured annulus, with $p$ as puncture. Then the
  subgroup $\Gamma$ of ${\rm Mod}(S)$ generated by the Dehn twists about the
  curves from the curve system ${\cal C}$ contains the
  kernel of the homomorphism
  ${\rm Mod}(S)\to {\rm Mod}(\Sigma)$ where $\Sigma$ is obtained from
  $S$ by forgetting $p$.
\end{lemma}

For a closed surface $\Sigma_g$ of genus $g\geq 2$ consider the 
system ${\cal S}_g$
of $3g-2$ simple closed curve on $\Sigma_g$ shown in Figure 4. 
\begin{figure}[ht]
\begin{center}
\includegraphics[width=0.5\textwidth]{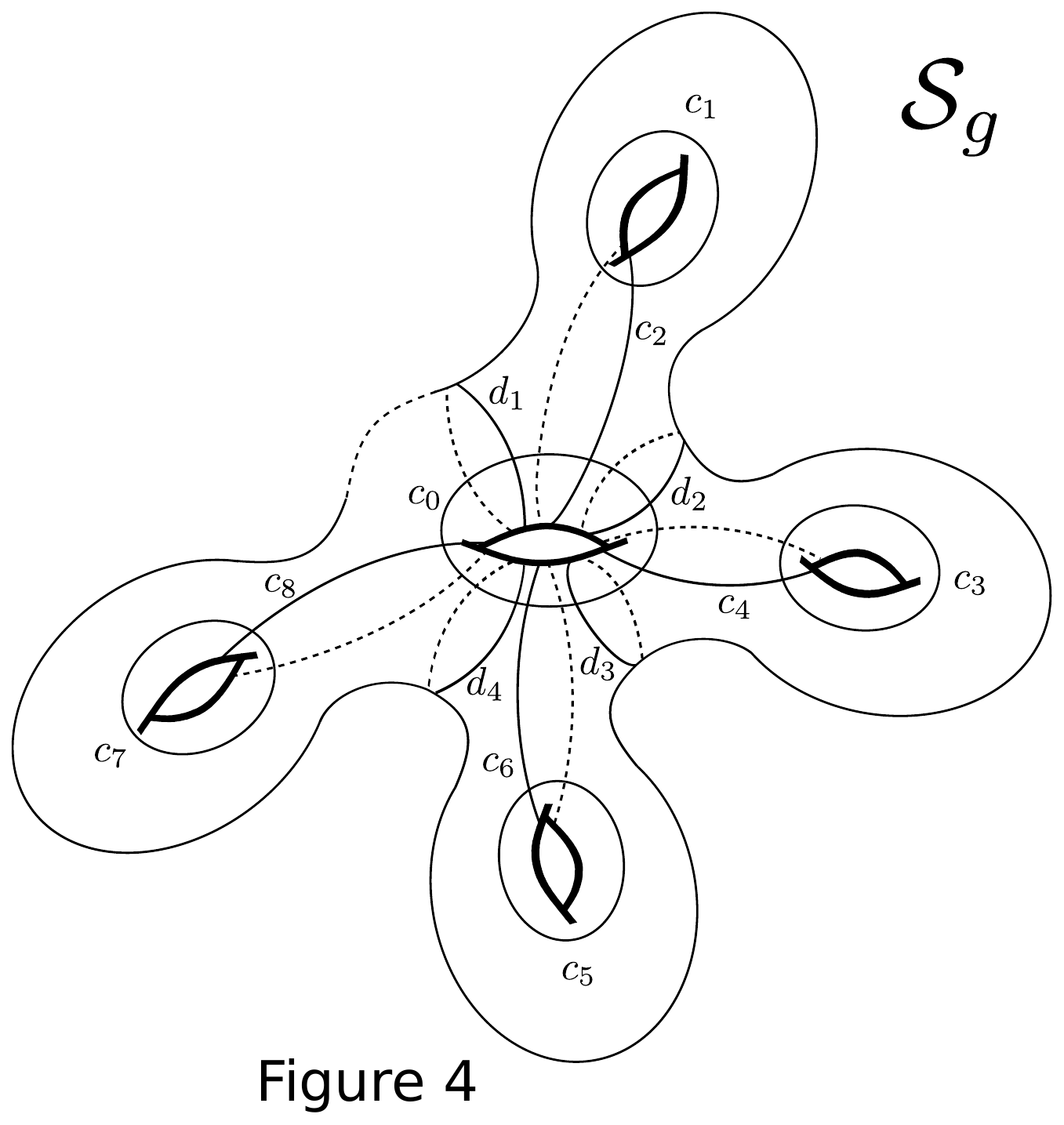}
\end{center}
\end{figure}

The following is well known but hard to locate in the literature.

\begin{lemma}\label{genus2}
For $g=2$, Dehn twists about the
curves from the curve system ${\cal S}_g$ 
generate the stabilizer of an odd spin structure in ${\rm Mod}(\Sigma_2)$.
\end{lemma}
\begin{proof}
${\cal S}_2$ is just a chain
of $4$ curves which are invariant under the hyperelliptic involution. 
The subgroup $\Gamma$ of ${\rm Mod}(\Sigma_2)$ generated by the 
 Dehn twists about these curves fixes one of the Weierstrass points, that is, one of  
the fixed points of the hyperelliptic involution.

Replacing the Weierstrass point by a small disk shows that we may view
$\Gamma$ as a subgroup of the quotient by its center of the 
\emph{symmetric mapping class group}  of 
a surface of genus $2$ with connected boundary, which equals the
braid group in 5 strands (Theorem 9.2  of \cite{FM12}). 
This group is known to be generated by the Dehn twists about the curves from the 
curve system ${\cal S}_2$ (see Section 9.1 of \cite{FM12}) and hence 
$\Gamma\subset {\rm Mod}(\Sigma_2)$ equals the stabilizer of one of the Weierstrass
points. This Weierstrass point defines a $\mathbb{Z}/2\mathbb{Z}$-spin structure on 
$\Sigma_2$ whose stabilizer in ${\rm Mod}(\Sigma_2)$ equals the stabilizer of the 
Weierstrass point.
This shows the lemma.
\end{proof}

 \begin{lemma}\label{preserve}
The Dehn twists about the curves from the system 
${\cal S}_g$ preserve an odd $\mathbb{Z}/2\mathbb{Z}$-spin
 structure on $\Sigma_g$. 
 \end{lemma}
 \begin{proof} There exists a cyclic subgroup $G$ of 
 the diffeomorphism group of $\Sigma_g$ 
 of order $g-1$ 
 which preserves ${\cal S}_g$ and acts freely on $\Sigma_g$ as a 
 group of rotations 
 about the center curve $c_0$. The group $G$  
 cyclically permutes the complementary
 components of ${\cal S}_g$. 
 
 As a consequence, the curve system ${\cal S}_g$ descends
 to a curve system on a closed 
 surface $\Sigma_2$ of genus $2$. The curve
 diagram of this system is just a line segment of length 4 and hence
 the Dehn twists about these curves preserve an odd spin structure
 on $\Sigma_2$ (see Lemma \ref{genus2}). 
 This spin structure lifts to a spin structure
 on $\Sigma_g$ which is invariant under the Dehn twist about the curves
 from ${\cal S}_g$. The parity of this spin structure is odd,
 as can also easily
 be checked explicitly using the formula (\ref{arf}).  
 This is what we wanted to show.
\end{proof}

We use Lemma \ref{generatepointpush} and 
Proposition \ref{spinmap}  to show

\begin{proposition}\label{cross}
Let $\phi$ be an odd $\mathbb{Z}/2\mathbb{Z}$-spin structure on a surface 
$\Sigma_g$ of genus $g\geq 2$. Then 
the group ${\rm Mod}(\Sigma_g)[\phi]$ is 
generated by the Dehn twists about the curves
from the curve system ${\cal S}_g$.
\end{proposition}
\begin{proof} 
Lemma \ref{preserve} shows that the 
subgroup $\Gamma$ of ${\rm Mod}(\Sigma_g)$ generated by the Dehn twists about the
curves from the curve system ${\cal S}_g$ 
is a subgroup of  
${\rm Mod}(\Sigma_g)[\phi]$.  
We have to show that it coincides
with ${\rm Mod}(\Sigma_g)[\phi]$.

We proceed by induction on the genus. 
By Lemma \ref{genus2}, the claim of the proposition holds true for $g=2$.
Thus let us assume that the proposition is known for some
$g-1\geq 2$. Consider the curve system ${\cal S}_g$ on a surface of genus
$g$. 
Using the labeling from Figure 4, let $a_1$ be the simple closed
curve on $\Sigma_g$ which intersects the curve $c_1$ in a single point
and is disjoint from any other curve from ${\cal S}_g$. We know that
$\phi(a_1)=1$. 
We aim at showing that
$\Gamma\cap {\rm Stab}(a_1)={\rm Mod}(\Sigma_g)[\phi]\cap
{\rm Stab}(a_1)$. 

To this end cut $\Sigma_g$ open along $a_1$. The resulting surface
is a surface $\Sigma_{g-1}^2$ of genus $g-1$ with two boundary
components. Replace these two boundary components by punctures and let
$\Sigma_{g-1,2}$ be the resulting twice punctured surface.
As before, the spin structure $\phi$ decends to a spin structure,
again denoted by $\phi$, on the surface $\Sigma_{g-1}$ obtained
by closing the punctures, and to a spin structure 
on $\Sigma_{g-1,2}$.
The curve system ${\cal S}_g$ descends to the curve system
${\cal S}_{g-1}$ on $\Sigma_{g-1}$.

By induction hypothesis, the Dehn twists about the curves from
the curve system ${\cal S}_{g-1}$ generate the spin mapping class
group ${\rm Mod}(\Sigma_{g-1})[\phi]$. On the other hand,
we can apply Lemma \ref{generatepointpush} to each of the two
punctures of $\Sigma_{g-1,2}$ as each of these two punctures
is contained in a once punctured annulus bounded by two
curves from the restriction of ${\cal S}_g$ to
$\Sigma_{g-1,2}$. We conclude that the point pushing
maps about these punctures are contained in
the group $\Gamma\cap {\rm Stab}(a_1)$. As a consequence,
the group $\Gamma\cap {\rm Stab}(a_1)$ surjects onto
the index two subgroup of ${\rm Mod}(\Sigma_{g-1,2})$
which fixes each of the two punctures.
 
We have to show that there also is an element of
$\Gamma\cap {\rm Stab}(a_1)$ which
exchanges the two boundary components of $\Sigma_g-a_1$. 
For this it suffices to find an 
element of $\Gamma$ which fixes 
the curves $c_1,c_2$ and exchanges $d_1,d_2$.

If $g=3$ then consider the 
hyperelliptic involution of the surface 
$\Sigma_2$ obtained by cutting $\Sigma_3$ open along the
simple closed curve $a_1$ and removing the punctures.
This element can be represented as an explicit word in the 
Dehn twists about the curve
$c_2,c_0,c_4,c_3$ (or, rather, their projection to $\Sigma_2$). 
The mapping class  $\psi$, viewed as an element of the
mapping class group of $\Sigma_3$, 
 preserves the curves $c_i$ and
exchanges $d_1$ and $d_2$.

For $g\geq 4$ the same argument can be used. 
Namely, the element $\psi$ still acts as an involution
on $\Sigma_g$ which preserves the curves $c_1,c_2$ and
exchanges $d_1$ and $d_2$. However this involution does
not preserve the curve system ${\cal S}_g$.

To summarize, we showed so far that
$\Gamma$ surjects onto ${\rm Stab}(a_1)[\phi]/\mathbb{Z}$. 
Thus to show that
$\Gamma\cap {\rm Stab}(a_1)={\rm Mod}(\Sigma_g)[\phi]\cap
{\rm Stab}(a_1)$ it suffices to show that
$\Gamma$ contains the square $T_{a_1}^2$ of the Dehn twist about $a_1$.
For an application of
Proposition \ref{spinmap}, we have to show furthermore that
$\Gamma$ contains an involution 
$\Psi$ which exchanges 
the curve $a_1$ with a curve disjoint from $a_1$.
We show first that $\Gamma$ contains an involution which
maps $a_1$ to $a_2$.

To this end consider again first the case $g=3$. 
The curve system ${\cal S}_3$ contains a curves system
${\cal E}_6
\subset {\cal S}_3$ obtained from  
${\cal S}_3$ by deleting the curve $d_2$. 
This is the curve system shown in Figure 2 in the introduction. 
By Theorem 1.4 of \cite{Ma00}, there exists 
an explicit word $c(E_6)$ 
in the Dehn twists about the curves from 
the system ${\cal E}_6$, the image of the so-called
\emph{Garside element} of the Artin group of type $E_6$, 
which acts 
as a reflection 
on the curve diagram of ${\cal E}_6$ exchanging the curves 
$c_1$ and $c_3$.
Then this reflection exchanges $a_1$ and $a_2$
and hence it has the desired properties. 

As before, this reasoning extends to any $g\geq 4$. 
Namely, the element $c(E_6)$, viewed as an element of the mapping
class group of $\Sigma_g$, still acts as an involution on
$\Sigma_g$ which exchanges $a_1$ and $a_2$ and preserves the
subsurface of $\Sigma_g$ filled by the curves $c_1,c_2,c_0,c_4,c_3,d_2$.

For an application of Proposition \ref{spinmap}, we are left with
showing that 
the square of the Dehn twist about $a_1$ is contained in 
$\Gamma$. By the above discussion, we know that
$\Gamma\cap {\rm Stab}(a_1)$ surjects onto
${\rm Mod}(\Sigma_{g-1,2})[\phi]$. In particular,
$\Gamma$ contains $T_{a_2}^2$, viewed as an element of
${\rm Stab}(a_1)\subset {\rm Mod}(\Sigma_g)$. Since
$a_1$ is the image of $a_2$ under an involution contained in
$\Gamma$, it follows that $T_{a_1}^2\in \Gamma$.

To summarize, we showed that $\Gamma\cap {\rm Stab}(a_1)=
{\rm Mod}(\Sigma_g)[\phi]\cap {\rm Stab}(a_1)$, furthermore 
$\Gamma$ contains an involution $\Psi$ which exchanges $a_1$ and $a_2$.
Proposition \ref{spinmap} now shows that $\Gamma={\rm Mod}(\Sigma_g)[\phi]$. 
This completes the proof of the Proposition.
\end{proof}


We use Proposition \ref{cross} as the base case for the proof of
Theorem \ref{main2} from the introduction. The curve system 
${\cal C}_g$ is shown in Figure 1 in the introduction.
Note that 
we have ${\cal C}_3={\cal S}_3$. 

\begin{theorem}\label{generatebycurves}
Let $\phi$ be an odd $\mathbb{Z}/2\mathbb{Z}$-spin structure 
on a surface $\Sigma_g$ of genus $g\geq 3$. Then the group 
${\rm Mod}(\Sigma_g)[\phi]$ 
is generated 
by the Dehn twists about the curves from the 
curve system ${\cal C}_g$.
\end{theorem}
\begin{proof}
The curve system ${\cal C}_g$ is obtained
from the curve system ${\cal S}_g$ by deleting the 
curves $d_3,\dots,d_{g-1}$. 
Let $\Gamma$ be the subgroup of ${\rm Mod}(\Sigma_g)[\phi]$
generated by the Dehn twists about the curves from the
curve system ${\cal C}_g$. By Proposition \ref{cross}, it suffices to show
that the Dehn twists $T_{d_i}$ for $i=3,\dots,g-1$ are contained in 
$\Gamma$. Moreover, as  
${\cal D}_3={\cal C}_3$,
we may assume
that $g\geq 4$.

Let $a_i$ be the simple closed curve which 
intersects $c_{2i-1}$ in a single point and does not intersect
any other curve from ${\cal S}_g$. 
We first claim that $T_{a_1}^2\in \Gamma$. 

To show the claim consider the subsurface
$\Sigma_2^1$ of $\Sigma_g$ which is filled by 
the curves $a_1,c_0,c_1,c_2,d_1,d_2$. This is a surface of genus
$2$ with connected boundary. 
The curves $d_1,d_2$ bound a one-holed annulus containing
the boundary circle $C$ of $\Sigma_2^1$.

By homological coherence (Proposition 3.8 of \cite{Sa19}),
we have $\phi(C)=1$. Thus the
spin structure $\phi$ descends to a spin structure on
$\Sigma_2^1$, on the surface $\Sigma_{2,1}$ obtained from 
$\Sigma_2^1$ by
replacing the boundary component by a puncture 
and on the surface
$\Sigma_2$ obtained from $\Sigma_{2,1}$ by forgetting the
puncture, again denoted by $\phi$. The curves 
of the curve  system ${\cal C}_g$ which are 
contained in $\Sigma_2^1$ 
define a curve system ${\cal F}$ on $\Sigma_2^1$ which 
descends to a curve system on $\Sigma_2$.
The curve diagram of this system is just a line segment of length 4.
By Lemma \ref{genus2}, the Dehn twists about the curves from 
${\cal F}$ project onto ${\rm Mod}(\Sigma_2)[\phi]$. 

On the other hand, ${\cal F}$ also contains two simple closed
curves which enclose the boundary component of $\Sigma_{2,1}$.
It now follows from Lemma \ref{generatepointpush} that
the subgroup of ${\rm Mod}(\Sigma_{2,1})$ generated by the 
Dehn twists about the curves from ${\cal F}$ equals
${\rm Mod}(\Sigma_{2,1})[\phi]$. In particular, this group contains 
$T_{a_1}^2$ and therefore 
$T_{a_1}^2\in \Gamma$.

We claim next that $T_{a_2}^2\in \Gamma$. To this end 
consider the subsurface 
$\Sigma_{3}^1$ of $\Sigma_g$ which is
filled by the system of 
curves ${\cal G}=\{c_1,c_2,c_0,c_4,c_3,d_1,d_2,d_3\}$. This is
a surface of genus $3$ with connected boundary.
The curves $d_1,d_3$ bound a one-holed annulus containing
the boundary circle $A$ of $\Sigma_3^1$.

The subsurface $\Sigma_3^1$ of $\Sigma_g$ 
contains the curves $c_1,c_2,c_0,c_4,c_3,d_2$ whose
curve diagram is the Dynkin diagram of type $E_6$
(see Figure 2 in the introduction).
There is an involution of $\Sigma_3^1$ 
which fixes the curves $c_0,d_2$ and 
exchanges $c_2,c_4$ and $a_1,a_2$. By Theorem 1.4 of 
\cite{Ma00}, this involution is contained in the subgroup of 
the mapping class group of $\Sigma_3^1$ which is generated by the
Dehn twists about the curves $c_1,c_2,c_0,c_4,c_5,d_2$. 
As a consequence, there is an element of $\Gamma$ which 
exchanges $a_1$ and $a_2$. This implies that
$T_{a_2}^2\in \Gamma$.

By the chain relation for Dehn twists of surfaces 
(see p.108 of \cite{FM12}), we have
$(T_{a_2}^2T_{c_3}T_{c_4})^3=T_{d_2}T_{d_3}$. Since
$T_{d_2}\in \Gamma$, we conclude that 
$T_{d_3}\in \Gamma$. 

Now repeat this argument, replacing the curves $c_j$ by
$c_{j+2}$ and the curve $a_i$ by $a_{i+1}$ 
where the first step discussed above is the case $i=1$.  
In finitely
many such steps we find that indeed 
$T_{d_i}\in \Gamma$ for all $i$. This is what we wanted to show.
\end{proof}

\section{Structure of the spin mapping class group of even parity}\label{structureeven}

The goal of this section is to prove the second part of Theorem \ref{main2}.
Our strategy is to reduce this result to the first part of Theorem \ref{main2} by
 a change of parity construction. 

Consider for the moment an
arbitrary
$\mathbb{Z}/r\mathbb{Z}$-spin structures $\phi$ on
a compact surface $S$ of genus $g\geq 4$.
%
In the appendix we introduce a graph ${\cal C\cal G}_2^+$ 
whose vertices are ordered pairs $(a,b)$ 
of nonseparating simple closed curves which intersect in a single
point and hence they fill a one-holed torus $T(a,b)$. Furthermore, it is 
required that $\phi(a)=2$ and $\phi(b)=0$. 
The spin structure on $S$ restricts to a spin structure $\hat \phi$
on $\Sigma(a,b)=S-T(a,b)$. 

By homological coherence (Proposition 3.5 of \cite{Sa19}), 
if we orient the boundary circle $c$ of $\Sigma(a,b)$ as the oriented
boundary of $\Sigma(a,b)$ then we have $\phi(c)=1$. Thus 
if $r=2$ then  
$\phi$ descends to a spin structure $\hat \phi$ 
on the surface $\Sigma$ obtained
from $\Sigma(a,b)$ by capping off the boundary. This spin structure
$\hat \phi$ 
has a parity, either even or odd.

\begin{lemma}\label{nextcase}
A $\mathbb{Z}/2\mathbb{Z}$-spin structure $\phi$ 
on $S$ induces a 
$\mathbb{Z}/2\mathbb{Z}$-spin structure $\hat \phi$ on 
the surface $\Sigma$ whose parity is opposite to the parity of $\phi$.
\end{lemma}
\begin{proof}
Choose a geometric symplectic basis
  $a_1,b_1,\dots, a_{g-1},b_{g-1}$ for 
  $\Sigma$. This basis then lifts to a curve system on the
  surface $\Sigma(a,b)=S-T(a,b)$. 
  Using the inclusion $\Sigma(a,b)\to S$, this basis can be extended
  to a geometric symplectic basis of $S$ by adding
  $a,b$. As $\phi(a)=\phi(b)=0$, the parity of 
  $\phi$ is opposite to the parity of $\hat \phi$.
\end{proof}

The next observation is an analog of Proposition \ref{connect}.
Note that we only require $g\geq 3$ here.

\begin{proposition}\label{connect2}
Let $\phi$ be a $\mathbb{Z}/r\mathbb{Z}$-spin structure on a compact surface 
$S$ of genus $g\geq 3$ with empty or connected boundary. 
Then for any two vertices 
$c,d$ of the graph ${\cal C\cal G}_2^+$
there exists a mapping class 
$\zeta\in {\rm Mod}(S)[\phi]$ with $\zeta(c)=d$.
In particular, the action of
  ${\rm Mod}(S)[\phi]$ is transitive on the vertices of the graph
 ${\cal C\cal G}_2^+$.
\end{proposition}  
\begin{proof} The proof is very similar to the proof 
of Proposition \ref{connect} and will be omitted.
\end{proof}

Consider again 
a $\mathbb{Z}/r\mathbb{Z}$-spin
structure $\phi$ on a closed surface $\Sigma_g$ of genus $g\geq 3$. 
Let $c$ be a separating simple closed curve on $\Sigma_g$ 
which is the boundary of a small neighborhood of a vertex
$(a,b)\in {\cal C\cal G}_2^+$. Then $c$ 
decomposes $\Sigma_g$ into a one holed torus $\Sigma_1^1$ and a surface 
$\Sigma_{g-1}^1$ of genus $g-1$ with connected boundary. 
The spin structure restricts to a spin structure on
$\Sigma_1^1$. If $r$ is even then this spin structure has a parity,
and this parity is odd.

Since $c$ is separating, 
the group ${\rm Mod}(\Sigma_{g-1}^1)[\phi]\times {\rm Mod}(\Sigma_1^1)[\phi]$  
contains 
a rank two free abelian central subgroup generated by the
left Dehn twists $T_{c_1},T_{c_2}$ 
about the 
boundary circles $c_1,c_2$ of $\Sigma_{g-1}^1,\Sigma_1^1$.  Define 
\[\Gamma_{g-1,2}^2={\rm Mod}(\Sigma_{g-1}^1)[\phi]
  \times {\rm Mod}(\Sigma_1^1)[\phi]/\mathbb{Z}\]
where the infinite cyclic subgroup $\mathbb{Z}$ is generated
by $T_{c_1}T_{c_2}^{-1}$. Then 
$\Gamma_{g-1,2}^2$ is isomorphic to the stabilizer 
in ${\rm Mod}(\Sigma_g)[\phi]$ of the curve $c$. 
Note that up to isomorphism, 
the group $\Gamma_{g-1,2}^2$ 
does not depend on $c$ since
by  Proposition \ref{connect2}, 
the stabilizers 
in ${\rm Mod}(\Sigma_g)[\phi]$
of vertices of ${\cal C\cal G}_2^+$  
are all conjugate and hence isomorphic. 
 
Observe that the group $\Gamma_{g-1,2}^2$ is an infinite cyclic central extension of 
the product of a finite index subgroup of the mapping class
group 
of a surface $\Sigma_{g-1,1}$ 
of genus $g-1$ with one puncture and a once punctured torus
$\Sigma_{1,1}$. Thus it makes sense
to talk about its action on isotopy classes of 
essential curves on the surfaces 
$\Sigma_{g-1,1}$ and $\Sigma_{1,1}$. The map $\Sigma_{g-1}^1\coprod
\Sigma_1^1\to 
\Sigma_{g-1,1}\coprod \Sigma_{1,1}$ which contracts
each boundary component to a puncture
defines a bijection on such isotopy classes.  

The following observation is the analog of Proposition \ref{spinmap}.

\begin{proposition}\label{spinmap2}
Let $\phi$ be a $\mathbb{Z}/r\mathbb{Z}$-spin structure on 
a closed surface $\Sigma_g$ of genus $g\geq 4$.  
There is a commutative diagram\\
\begin{equation}
\begin{tikzcd}
\Gamma_{g-1,2}^2 \arrow[r,"\iota_1"] \arrow[dr, "\iota_2"]
& \Gamma_{g-1,2}^2 *_A\Gamma_{g-1,2}^2\rtimes \mathbb{Z}/2\mathbb{Z}
\arrow[d, "\rho"] \\
& {\rm Mod}(\Sigma_g)[\phi]
\end{tikzcd} 
\end{equation}
where the homomorphisms $\iota_1,\iota_2$ are inclusions, and the 
homomorphism $\rho$ is surjective. 
The subgroup $A$ of $\Gamma_{g-1,2}^2$ 
is the stabilizer in $\Gamma_{g-1,2}^2$ 
of a 
separating simple closed curve $d$ 
on $\Sigma_{g-1}^2$ which is defined by a 
vertex of the graph ${\cal C\cal G}_2^+$. The curve $d$  
decomposes $\Sigma_{g-1}^1$ into a one-holed torus
and a surface of genus $g-2$ with two boundary components. 
The group $\mathbb{Z}/2\mathbb{Z}$ acts on $\Gamma_{g-1,2}^2*_A\Gamma_{g-1,2}^2$
by exchanging the two factors, and it acts as an automorphism
on $A$.  
\end{proposition}
\begin{proof}
Fix a pair of vertices of the graph ${\cal C\cal G}_2^+$ which
are connected by an edge. These two vertices then determine a pair of 
disjoint separating simple closed curves 
$c,d$ on $\Sigma_g$ which cut from $\Sigma_g$ a one-holed torus each.
These tori are disjoint.
Let $\Gamma_c,\Gamma_d\subset {\rm Mod}(\Sigma_g)[\phi]$ be the stabilizers of 
$c,d$ in the spin mapping class group of $\Sigma_g$. By Corollary \ref{twoconnect}, 
these groups 
are naturally isomorphic to the group $\Gamma_{g-1,2}^2$, and they
intersect in the index two subgroup $A=\Gamma_c\cap \Gamma_d$ 
of the stabilizer
of $c\cup d$ in 
${\rm Mod}(\Sigma_g)[\phi]$ consisting of all elements which preserve both $c,d$
individually. 
The full stabilizer of $c\cup d$ in ${\rm Mod}(\Sigma_g)[\phi]$ 
is a $\mathbb{Z}/2\mathbb{Z}$
extension of $\Gamma_c\cap \Gamma_d$, where the generator 
$\Phi$ of $\mathbb{Z}/2\mathbb{Z}$ acts as involution on 
$A=\Gamma_c\cap \Gamma_d$ exchanging $c$ and $d$. This involution 
extends to an involution of $\Gamma_c*_A\Gamma_d$ 
exchanging the two subgroups $\Gamma_c,\Gamma_d$.

By the universal property of free amalgamated products, 
there is 
a homomorphism 
\[\rho:\Gamma=\Gamma_c*_A\Gamma_d\rtimes \mathbb{Z}/2\mathbb{Z}
\to {\rm Mod}(\Sigma_g)[\phi]. \]
All we need to show is that $\rho$ is surjective, that is, that 
$\rho(\Gamma)= 
{\rm Mod}(\Sigma_g)[\phi]$.

As ${\rm Mod}(\Sigma_g)[\phi]$ acts transitively on the vertices of the graph 
${\cal C\cal G}_2^+$, for this it suffices to show that its subgroup 
$\rho(\Gamma)$ acts transitively on the vertices of ${\cal C\cal G}_2^+$
as well. Namely, by construction, the stabilizer of the vertex
$c$ of ${\cal C\cal G}_1^+$ in $\rho(\Gamma)$ coincides with its
stabilizer in ${\rm Mod}(\Sigma_g)[\phi]$. As $\rho(\Gamma)$ is a subgroup
of ${\rm Mod}(\Sigma_g)[\phi]$, this then implies equality.

To show transitivity of the action of $\rho(\Gamma)$ on the vertices
of ${\cal C\cal G}_2^+$ let $v\in {\cal C\cal G}_2^+$ be any vertex.
By Proposition \ref{connected2}, 
the graph ${\cal C\cal G}_2^+$ is connected and hence we can find an edge path
$(c_i)\subset {\cal C\cal G}_2^+$ connecting $c_0=c$ to $c_k=v$. We also may assume that
$c_1=d$.

By Proposition \ref{connect2}, 
there exists an element 
$\Psi_1\in \Gamma_d\subset \rho(\Gamma)$ such that
$\Psi_1(c_0)=c_2$. Then the stabilizer of 
$c_2$ in ${\rm Mod}(\Sigma_g)[\phi]$ equals $\Psi_1\Gamma_c\Psi_1^{-1}$ and hence
it is contained in $\rho(\Gamma)$. 
Thus we can apply 
Corollary \ref{twoconnect} to $\Psi_1\Gamma_c\Psi_1^{-1}$ and find an element
$\Psi_2\in \rho(\Gamma)$ which maps $c_1$ to $c_3$. 
Proceeding inductively and using the fact
that $\Gamma_c$ is conjugate to $\Gamma_d$ in $\rho(\Gamma)$ 
by the generator of the subgroup $\mathbb{Z}/2\mathbb{Z}$,  
this completes the proof of the proposition.
\end{proof}

For a surface $S$ of genus $g\geq 3$ consider the following
system ${\cal U}_g$
of $3g-2$ simple closed curve on $S$. 
\begin{figure}[ht]
\begin{center}
\includegraphics[width=0.6\textwidth]{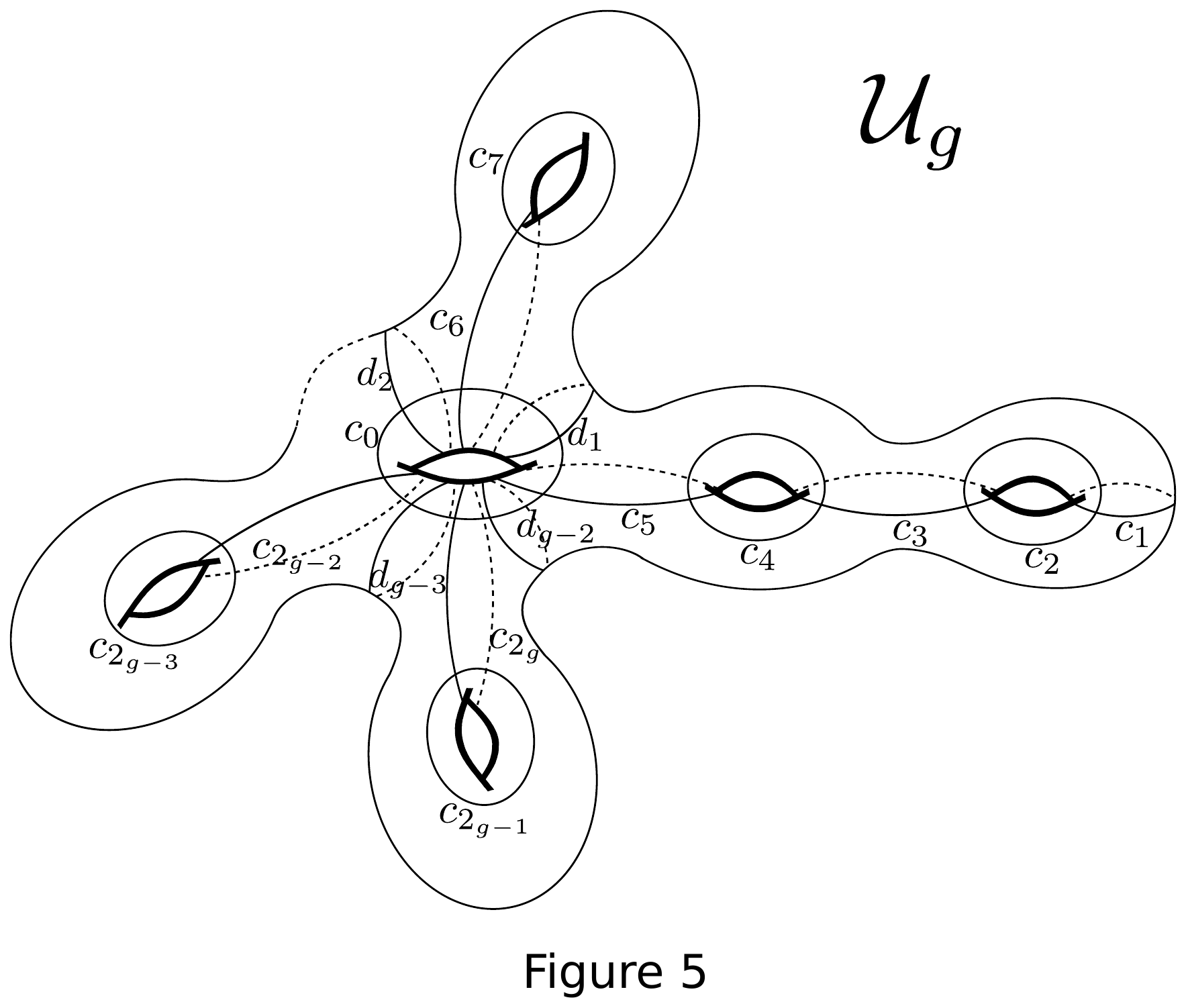}
\end{center}
\end{figure}
Note that for $g=3$, the system ${\cal S}_g$ is just a chain
of $7$ curves which are invariant under a hyperelliptic involution.
It follows from Lemma \ref{preserve} and Lemma \ref{nextcase} 
that the Dehn twists about these curves preserve an even
$\mathbb{Z}/2\mathbb{Z}$-spin structure on $\Sigma_g$.

We use Lemma \ref{generatepointpush} and 
Proposition \ref{spinmap}  to show

\begin{proposition}\label{cross2}
Let $\phi$ be an even $\mathbb{Z}/2\mathbb{Z}$-spin structure on a surface 
$\Sigma_g$ of genus $g\geq 4$. Then 
the group ${\rm Mod}(\Sigma_g)[\phi]$ is 
generated by the Dehn twists about the curves
from the curve system ${\cal U}_g$.
\end{proposition}
\begin{proof} 
We observed above that the 
subgroup $\Gamma$ of ${\rm Mod}(\Sigma_g)$ generated by the Dehn twist about the
curves from the curve system ${\cal U}_g$ 
is a subgroup of  
${\rm Mod}(\Sigma_g)[\phi]$.  
We have to show that it coincides
with ${\rm Mod}(\Sigma_g)[\phi]$.

To this end we proceed by induction on the genus, beginning with genus 4.
Let $a$ be the separating simple closed curve which intersects $c_3$
in two points and is disjoint from the remaining curves from the system
${\cal U}_4$. It decomposes ${\cal U}_4$ into a one holed torus
$\Sigma_1^1$ containing the curves $c_1,c_2$, and a surface $\Sigma_3^1$ 
of genus $3$ with connected
boundary which contains the curve system ${\cal S}_3$.
As we are looking at a $\mathbb{Z}/2\mathbb{Z}$-spin structure
we know that the pair $(c_1,c_2)$ and hence the curve $a$ 
defines a vertex in ${\cal C\cal G}_2^+$. 
The spin structure $\phi$
induces a spin structure on $\Sigma_3^1$ and $\Sigma_1^1$, again denoted by $\phi$.
It also induces a spin structure on the closed surface $\Sigma_3$ of genus $3$ 
obtained from $\Sigma_3^1$ by capping off the boundary, again denoted
by $\phi$.

It is well known that the mapping class group of one holed tori is 
generated by a pair of Dehn twists about simple closed curves which intersect in a single point.
Thus we have ${\rm Mod}(\Sigma_1^1)\subset \Gamma\cap {\rm Stab}(a)$.

On the other hand, by Proposition \ref{cross}, the Dehn twists about the curves from the
system ${\cal S}_3$ generate the spin mapping class group
${\rm Mod}(\Sigma_3)[\phi]$ of $\Sigma_3$. Thus the projection of 
$\Gamma$ to ${\rm Mod}(\Sigma_3)[\phi]$ is surjective.

To apply Proposition \ref{spinmap2} we have to show 
that the point pushing group of ${\rm Mod}(\Sigma_{3,1})[\phi]$ 
is contained in the projection of $\Gamma\cap 
{\rm Stab}(a)$. We use once more Lemma \ref{generatepointpush}
to this end.

Consider the curves 
$c_0,c_7,c_6,d_1,c_5$ which define a curve system on 
the surface $\Sigma_{3,1}$ whose curve diagram is the Dynkin diagram $D_5$. 
By Theorem 1.5 of \cite{Ma00}, there exists
an explicit word in the Dehn twists about these curves which defines the 
product $T_{a_4}^3T_{a_4^\prime}$ where
$a_4$ is simple closed curve in $\Sigma_3^1$
which intersects $c_4$ in a single point 
and is disjoint from all other curves and where $a_4^\prime$ is the simple closed curve 
which bounds together with $a_4$ a once punctured annulus in $\Sigma_3^1$.

On the other hand, 
the chain relation \cite{FM12} yields that 
$T_{a_4}^\prime T_{a_4}=(T_{c_1}T_{c_2}T_{c_3})^4$.
 Since $T_{a_4},T_{a_4^\prime}$ commute we deduce 
that $(T_{a_4}^\prime)^{-2}T_{a_4}^{-2}T_{a_4}^3T_{a_4^\prime}=
T_{a_4}T_{a_4^\prime}^{-1}\in \Gamma$.
As a consequence, the group $\Gamma\cap {\rm Stab}(a)$ contains
the point pushing map $T_{a_4}T_{a_4^\prime}^{-1}$ about the based loop  
$\alpha$ in $\Sigma_{3}$ which is homotopic to the common projection of
$a_4,a_4^\prime$. Lemma \ref{generatepointpush} now shows that
$\Gamma\cap {\rm Stab}(a)$ contains indeed the point pushing
group of ${\rm Mod}(\Gamma_{3,1})[\phi]$.

By Proposition \ref{spinmap2}, we are left with finding an element
$\Psi\in \Gamma$ which maps $a$ to a curve disjoint from $a$.
However, the curve system ${\cal U}_4$ contains
a subsystem consisting of the curves $c_i$ $(i=0,\dots, 7)$.
The Dehn twists about these curves are well known to generate
the stabilizer of a Weierstrass point in the 
\emph{hyperelliptic mapping class group}, that is, the
subgroup of the mapping class group which commutes with
a hyperelliptic involution (\cite{FM12} and compare Lemma \ref{genus2}). 
This group is isomorphic to the quotient of the Artin braid group in
$2g+1$ strands by its center, and it 
contains an element $\psi$ which maps $a$ to a disjoint curve,
e.g. the boundary of a small neighborhood of
$c_0\cup c_5$. The proposition for $g=4$ now follows from
Proposition \ref{spinmap2}.

By induction, 
let us now assume that the proposition is known for some
$g-1\geq 4$. Consider the curve system ${\cal U}_g$ on a surface of genus
$g$. 
Using the labeling from Figure 5, let $a_7$ be the simple closed
curve on $\Sigma_g$ which intersects the curve $c_7$ in a single point
and is disjoint from any other curve from ${\cal U}_g$. We know that
$\phi(a_7)=1$. 
We aim at showing that
$\Gamma\cap {\rm Stab}(a_7)={\rm Mod}(\Sigma_g)[\phi]\cap
{\rm Stab}(a_7)$. 

To this end cut $\Sigma_g$ open along $a_7$. The resulting surface
is a surface $\Sigma_{g-1}^2$ of genus $g-1$ with two boundary
components. Replace these two boundary components by punctures and let
$\Sigma_{g-1,2}$ be the resulting twice punctured surface.
As before, the spin structure $\phi$ descends to a spin structure,
again denoted by $\phi$, on the surface $\Sigma_{g-1}$ obtained
by closing the punctures, and to a spin structure 
on $\Sigma_{g-1,2}$.
The curve system ${\cal U}_g$ descends to the curve system
${\cal U}_{g-1}$ on $\Sigma_{g-1}$.

By induction hypothesis, the Dehn twists about the curves from
the curve system ${\cal U}_{g-1}$ generate the spin mapping class
group ${\rm Mod}(\Sigma_{g-1})[\phi]$. On the other hand,
we can apply Lemma \ref{generatepointpush} to each of the two
punctures of $\Sigma_{g-1,2}$ as each of these two punctures
is contained in a once punctured annulus bounded by two
curves from the restriction of ${\cal U}_g$ to
$\Sigma_{g-1,2}$. We conclude that the point pushing
maps about these punctures are contained in
the group $\Gamma\cap {\rm Stab}(a_7)$. As a consequence,
the group $\Gamma\cap {\rm Stab}(a_7)$ surjects onto
${\rm Mod}(\Sigma_{g-1,1})$.

To summarize, we showed so far that
$\Gamma$ surjects onto ${\rm Stab}(a_7)[\phi]/\mathbb{Z}$
where $\mathbb{Z}$ is the intersection of 
${\rm Mod}(\Sigma_g)[\phi|$ with the infinite cyclic group of 
Dehn twists about $a_7$. 
Thus to show that
$\Gamma\cap {\rm Stab}(a)={\rm Mod}(\Sigma_g)[\phi]\cap
{\rm Stab}(a)$ it suffices to show that
$\Gamma$ contains the square $T_{a_7}^2$ of the Dehn twist about $a_7$
as well as an involution $\Psi$ which exchanges $a_7$ with a simple closed
curve disjoint from $a_7$. 

To find an involution $\Psi$ as required, 
consider first the case $g=4$. 
The curve system ${\cal U}_4$ contains a curve system
${\cal E}_6
\subset {\cal U}_4$ consisting of the curves $c_7,c_6,c_0,d_1,c_5,c_4$.
By Theorem 1.4 of \cite{Ma00}, there exists 
an explicit word $c(E_6)$ 
in the Dehn twists about the curves from 
the system ${\cal E}_6$, the image of the so-called
\emph{Garside element} of the Artin group of type $E_6$, 
which acts as a reflection 
on the curve diagram of ${\cal E}_6$ exchanging the curves 
$c_7$ and $c_4$.
Then this reflection maps $a_7$ to a disjoint curve $a_7^\prime$ 
and hence it has the desired properties. 

This reasoning extends to any $g\geq 5$. 
Namely, the element $c(E_6)$, viewed as an element of the mapping
class group of $\Sigma_g$, still acts as an involution on
$\Sigma_g$ which maps $a_7$ to a disjoint curve $a_7^\prime$  
and preserves the
subsurface of $\Sigma_g$ filled by the curves $c_7,c_6,c_0,d_1,c_5,c_4$.
Thus there always exists an involution $\Psi\in \Gamma$ which maps 
$a_7$ to a disjoint curve $a_7^\prime$. 

For an application of Proposition \ref{spinmap}, we are left with
showing that 
the square of the Dehn twist about $a_7$ is contained in 
$\Gamma$. By the above discussion, we know that
$\Gamma\cap {\rm Stab}(a_7)$ surjects onto
${\rm Mod}(\Sigma_{g-1,2})[\phi]$. In particular,
$\Gamma$ contains $T_{a_7^\prime}^2$, viewed as an element of
${\rm Stab}(a_7)\subset {\rm Mod}(\Sigma_g)$. Since
$a_7$ is the image of $a_7^\prime$ under an involution contained in
$\Gamma$, it follows that $T_{a_7}^2\in \Gamma$.

To summarize, we showed that $\Gamma\cap {\rm Stab}(a_7)=
{\rm Mod}(\Sigma_g)[\phi]\cap {\rm Stab}(a_7)$, furthermore 
$\Gamma$ contains an involution $\Psi$ which exchanges $a_7$ and $a_7^\prime$.
Proposition \ref{spinmap} now shows that $\Gamma={\rm Mod}(\Sigma_g)[\phi]$. 
This completes the proof of the Proposition.
\end{proof}


We use Proposition \ref{cross} as the base case for the proof of
the second part of 
Theorem \ref{main2} from the introduction. The curve system 
${\cal V}_g$ is defined as in the Theorem \ref{main2}. Note that 
we have ${\cal V}_3={\cal U}_3$. 

\begin{theorem}\label{generatebycurves2}
Let $\phi$ be an even $\mathbb{Z}/2\mathbb{Z}$-spin structure 
on a surface $\Sigma_g$ of genus $g\geq 4$. Then the group 
${\rm Mod}(\Sigma_g)[\phi]$ 
is generated 
by the Dehn twists about the curves from the 
curve system ${\cal V}_g$.
\end{theorem}
\begin{proof}
The curve system ${\cal V}_g$ is obtained
from the curve system ${\cal U}_g$ by deleting the
curves $d_2,\dots,d_{g-2}$. 

Let $\Gamma$ be the subgroup of ${\rm Mod}(\Sigma_g)[\phi]$
generated by the Dehn twists about the curves from the
curve system ${\cal V}_g$. By Proposition \ref{cross}, it suffices to show
that the Dehn twists $T_{d_i}$ for $i=2,\dots,g-2$ are contained in 
$\Gamma$. 

To see that $T_{d_{g-2}}\in \Gamma$, note that
$d_{g-2}$ is the image of $d_1$ under the hyperelliptic involution of 
the surface of genus $3$ with connected boundary filled by
the curves $d_1,d_{g-2},c_0,c_1,c_2,c_3,c_4$.

Consider the surface $S$  filled by $c_1,\dots,c_6,d_1,d_2,d_{g-2} $. 
This is a surface of genus $4$ with connected boundary.
The union of the system ${\cal V}_g$ with the curve $d_{g-2}$ 
intersects $S$ in a curve system 
of type ${\cal U}_4$. By Proposition \ref{cross2} and what we have proved so far, 
the stabilizer of the surface $S$ in the 
group $\Gamma$ surjects onto the
spin mapping class group of the surface obtained from $S$ by capping off the boundary.
In particular, if we denote by $e_1,e_3$ the nonseparating simple closed
curves which intersect $c_4$ in a single point, do not intersect any other curve
and form a bounding pair, then $T_{e_1}T_{e_2}^{-1}\in \Gamma$. 

Now by Theorem 1.4 of \cite{Ma00}, the stabilizer in $\Gamma$ of the surface of 
genus 3 with two boundary components obtained from $S$ by removing the 
one-holed torus $T$ filled by $c_1,c_2$ contains 
a half-twist which exchanges the two boundary components of the surface.
Let $S^\prime$ be the surface obtained from 
$S-T$ by replacing two boundary components by punctures. 
Lemma \ref{generatepointpush}, applied to the Dehn twists about the curves
$e_1,e_2$ enclosing the boundary component of $T$, shows that 
the subgroup of the mapping class group of $S^\prime$ which is the 
point pushing group of the puncture corresponding to 
the boundary of $T$ is contained in the projection of the stabilizer of 
$S-T$ in $\Gamma$. 
But then the same holds true for the point pushing 
group of the second puncture of $S^\prime$. 

This shows that we have $T_{d_2}T_{d_{g-2}}^{-1}\in \Gamma$. As 
$T_{d_{g-2}}\in \Gamma$, we conclude that the same holds true for $T_{d_2}$.
To generate the remaining twists about the curves $d_i$ we argue as in the
proof of Theorem \ref{generatebycurves}, using the Dehn twists
$T_{d_1}$ and $T_{d_2}$. 
\end{proof}

\section{Generating the $\mathbb{Z}/4\mathbb{Z}$-spin
  mapping class group in genus 3}\label{special}

The goal of this section
is to prove Theorem \ref{main3} from the
introduction.
Our strategy is similar
to the strategy used in 
Section  \ref{structureof}. We first introduce one more
graph of curves which will be
useful to this end. 

Consider an odd $\mathbb{Z}/2\mathbb{Z}$-spin structure $\phi$ on a 
surface $\Sigma_3$ of genus 3.
A separating simple closed curve $a$ on $\Sigma_3$ 
decomposes $\Sigma_3$ into a 
one-holed torus $T$ and a surface $\Sigma_2^1$ of genus 2 with connected boundary. 
By homological
coherence (Proposition 3.15 of \cite{Sa19}), we have $\phi(a)=1$. 
In particular,
$\phi$ induces a 
spin structure on the surface $\Sigma_2^1$ which has a parity. Define $a$ to be 
\emph{odd} if this parity is odd. Note that
a vertex of the graph ${\cal C\cal G}_2^+$ defined in the appendix and used
in Section \ref{structureeven} 
defines a separating
simple closed curve which is \emph{even}, that is, it is not odd.

Let ${\cal O\cal S}$ be the graph whose vertices are odd separating simple closed
curves on $(\Sigma_3,\phi)$ and where two such curves are connected by an edge
if they are disjoint. 
Let $\Phi$ be a $\mathbb{Z}/4\mathbb{Z}$-spin structure on $\Sigma_3$ whose
$\mathbb{Z}/2\mathbb{Z}$-reduction equals $\phi$. The 
stabilizer 
${\rm Mod}(\Sigma_3)[\phi]$ and its subgroup ${\rm Mod}(\Sigma_3)[\Phi]$ 
act on ${\cal O\cal S}$ as a group of simplicial automorphisms. 
The following observation is similar to Proposition \ref{connect}.
It uses some special properties of $\mathbb{Z}/4\mathbb{Z}$-spin
structures.

\begin{lemma}\label{transitive4}
  \begin{enumerate}
    \item 
      The group ${\rm Mod}(\Sigma_3)[\Phi]$ acts
      transitively on the vertices of ${\cal O\cal S}$.
    \item Let $a\in {\cal O\cal S}$ be any vertex. Then the stabilizer of
      $a$ in ${\rm Mod}(\Sigma_3)[\Phi]$ acts transitively on the
      edges of ${\cal O\cal S}$ issuing from $a$.
 \end{enumerate}       
\end{lemma}
\begin{proof}
A vertex $a$ of ${\cal O\cal S}$ decomposes $\Sigma_3$ into 
a one-holed torus $T$ and a surface $\Sigma_3-T$ of genus 2 with connected boundary
and odd spin structure. Since the parity of the spin structure of $\phi$ on $\Sigma_3$ is odd, 
the torus $T$ contains a simple closed curve 
$c$ with $\phi(c)=1$ and hence $\Phi(c)=\pm 1$. Via perhaps changing the orientation of
$c$ we may assume that $\Phi(c)=1$, furthermore there is a simple closed curve $d$ 
in $T$ which intersects $c$ in a single point and satisfies $\Phi(d)=0$. 

By homological coherence (Proposition 3.15 of \cite{Sa19}), if we orient $a$ as 
the oriented boundary of the surface $V=\Sigma_3-T$ then we have 
$\Phi(a)=1$. Since the spin structure induced on $V$ is odd, a geometric
symplectic basis for $V$ consists of simple closed curves
$a_1,b_1,a_2,b_2$ with $\phi(a_1)=1$ and hence $\Phi(a_1)=\pm 1$
(up to ordering). 
A tubular neighborhood $T^\prime$ of 
$a_1\cup b_1$ is an embedded bordered torus in $V$. Choose an orientation
for $a_1$ so that $\Phi(a_1)=1$. After
perhaps replacing $b_1$ by its image under a multiple of 
a Dehn twist about
$a_1$ we may assume that $\Phi(b_1)=0$. 

Consider the pair of curves $a_2,b_2$. Since the spin structure on $V$ is odd, we have
$\phi(a_2)=\phi(b_2)=0$ and hence $\Phi(a_2),\Phi(b_2)\in \{0,2\}$. Our goal is to 
modify $a_2,b_2$ so that $\Phi$ vanishes on the modified curves.
Thus assume without loss of generality that
$\Phi(a_2)=2$. 
Connect $a_2$ to the boundary curve $a$ of $V$ by an embedded arc
$\epsilon$ which is disjoint from $T^\prime$ and $b_2$, and connect $b_2$ to the boundary 
$\delta$ of 
$T^\prime$ by an embedded arc $\eta$ which is disjoint from $\epsilon$ and $a_2$.  
Since $\Phi(a)=1$ for the orientation 
as a boundary curve of $V$, we obtain that $\Phi(a_2+_\epsilon a)=0$, furthermore
this curve is disjoint from $T^\prime$ and intersects $b_2$ in a single point.
Replace $a_2$ by $a_2+_\epsilon a$. 
Similarly, if $\Phi(b_2)=2$ then we replace $b_2$ by $b_2+_\eta \delta$. 
This process yields a geometric symplectic basis for $\Sigma_3$ consisting of 
simple closed curves disjoint from $a$.

Given any 
other odd separating curve $a^\prime$ on $\Sigma_3$ we can find in the same way
a geometric symplectic basis for $\Sigma_3$ consisting of curves disjoint from $a^\prime$.
Then there is a mapping class which maps $a$ to $a^\prime$ and identifies the 
geometric symplectic bases in such a way that the values of $\Phi$ on these curves
match up. By the result of Humphries and Johnson \cite{HJ89}, this implies that this 
mapping class is contained in ${\rm Mod}(\Sigma_3)[\Phi]$. 
In other words, there is an element of ${\rm Mod}(\Sigma_3)[\Phi]$ which maps 
$a$ to $a^\prime$. This shows the first part of the lemma.

The proof of the second part of the lemma is completely analogous
but easier and will be omitted.
\end{proof}

\begin{lemma}\label{graphconnected3}
The graph ${\cal O\cal S}$ is connected.
\end{lemma}
\begin{proof}
Consider the curve system ${\cal C}_3$ on the surface $\Sigma_3$. 
Using the labels for the curves shown in Figure 4, 
there is an odd separating simple closed curve
$a$ which intersects the curve $c_2$ in two points and is disjoint from the remaining
curves from the system ${\cal C}_3$.
Using the Putman trick, Theorem \ref{generatebycurves}
and the first part of Lemma \ref{transitive4}, 
 all we need to show is that the curve $a$ can be connected to $T_{c_2}(a)$ by an edge path
 in ${\cal O\cal S}$.
 
 However, the curve $a^\prime$ which intersects the curve $c_4$ in two points and is
 disjoint from the remaining curves from the system ${\cal C}_3$ is separating and odd, 
 and it is disjoint from both $a$ and $T_{c_2}(a)$. Thus $a,a^\prime,T_{c_2}(a)$ is
 an edge path in ${\cal O\cal S}$ which connects $a$ to $T_{c_2}(a)$. 
\end{proof}

Using the labels from Figure 2 from the introduction, 
let $d$ be the separating simple closed curve on $\Sigma_3$ which 
intersects the curve $c_2$ in two points and is disjoint from the remaining
curves from the system ${\cal E}_6$. We show

\begin{lemma}\label{suffices}
The subgroup $\Gamma$ of ${\rm Mod}(\Sigma_3)$ which is generated by
the Dehn twists about the curves from the curve system ${\cal E}_6$ 
equals the stabilizer ${\rm Mod}(\Sigma_3)[\Phi]$
of an odd $\mathbb{Z}/4\mathbb{Z}$-spin
structure $\Phi$ on $\Sigma_3$ if and only if
its intersection with ${\rm Stab}(d)$ coincides with
${\rm Stab}(d)\cap {\rm Mod}(\Sigma_3)[\Phi]$. 
\end{lemma}
\begin{proof} Since $\Gamma$ is a subgroup of
  ${\rm Mod}(\Sigma_3)[\Phi]$, 
  the condition is clearly necessary,
  so we have to show sufficiency. Thus assume that
  $\Gamma\cap {\rm Stab}(d)={\rm Mod}(\Sigma_3)[\Phi]\cap {\rm Stab}(d)$.
  
Consider again the graph 
${\cal O\cal S}$.
Lemma \ref{graphconnected3} shows that ${\cal O\cal S}$  
 is connected. 
Moreover, by Lemma \ref{transitive4}, 
the group ${\rm Mod}(\Sigma_3)[\Phi]$
acts transitively on the directed edges of
${\cal O\cal S}$ as a group of simplicial
automorphisms. The curve $d$ is odd and
hence a vertex of ${\cal O\cal S}$.

By Theorem 1.4 of \cite{Ma00}, the group $\Gamma$ contains
an involution which induces a reflection 
in the curve diagram of the curve system ${\cal E}_6$ 
at the edge connecting the vertices $c_0$ and $c_3$. It 
maps the simple closed
curve $d$ to the separating simple closed curve $d^\prime$
which intersects $c_4$ in two points and 
is disjoint from all other curves from the system. 
Since $d$ is odd, 
the same is true for $d^\prime$.

We use this as follows. 
Let $e$ be any vertex of ${\cal O\cal S}$ 
and let $d=d_0,d_1,d_2,\dots,d_m=e$ be an
edge path in ${\cal O\cal S}$ 
which connects $d$ to $e$.
We may assume that $d_1=d^\prime$. 
Since there exists an element of $\Gamma$ which maps $d$ to $d^\prime$, 
the stabilizer of $d^\prime$ in $\Gamma$ is conjugate
to the stabilizer of $d$ and hence by our assumption, 
it coincides with the stabilizer of $d^\prime$ in 
${\rm Mod}(\Sigma_3)[\Phi]$. In particular,
by the second part of Lemma \ref{transitive4}, 
there exists an element of $\Gamma$ which fixes
$d^\prime$ and maps $d_0$ to $d_2$. 
Arguing inductively
as in the proof of Proposition \ref{spinmap},
we conclude that $\Gamma$ acts transitively
on the odd separating curves in $\Sigma_3$. 
As $\Gamma$ is a subgroup of ${\rm Mod}(\Sigma_3)[\Phi]$ and furthermore 
the stabilizer of a vertex in $\Gamma$ coincides with its
stabilizer in ${\rm Mod}(\Sigma_3)[\Phi]$,
it has to coincide with ${\rm Mod}(\Sigma_3)[\Phi]$.
The lemma follows.
\end{proof}

Our final goal is to show that the group $\Gamma$ fulfills the assumption in 
Lemma \ref{suffices}, which completes the proof of Theorem \ref{main3}.
We proceed in two steps.

Let $a_1,a_5$ be the nonseparating simple closed curves on $\Sigma_3$ which intersect
$c_1,c_5$ in a single point and are disjoint from the remaining curves from the 
system ${\cal E}_6$.
We have $\Phi(a_j)=\pm 1$, in particular, by Lemma 3.13 of \cite{Sa19}, 
the 
intersection of ${\rm Mod}(\Sigma_3)[\Phi]$ with the infinite cyclic group of Dehn twists
about the curve $a_j$ is generated by $T_{a_j}^4$.

\begin{lemma}\label{dehntwist4}
For $j=1,5$, the group $\Gamma$ contains $T_{a_j}^4$.
\end{lemma}
\begin{proof} Consider the 
  subsystem 
  ${\cal D}_5^j$ $(j=1,5)$
  obtained from the curve system ${\cal E}_6$ 
  by removing 
the curve $c_j$.  
By Theorem 1.3 (d) of \cite{Ma00},
the mapping class $T_{a_j}^4$ can be represented as 
an explicit
word in the Dehn twists about the curves from this
curve system.
Thus we have $T_{a_j}^4\in \Gamma$. 
\end{proof}

Lemma \ref{dehntwist4} is used in the proof of the final step towards
Theorem \ref{main3}.

\begin{lemma}\label{fullboundary}
The stabilizer in $\Gamma$ of the curve $d$ coincides with the 
stabilizer of $d$ in ${\rm Mod}(\Sigma_3)[\Phi]$. 
\end{lemma}
\begin{proof}
Let $T$ be the one-holed torus component of $\Sigma_3-d$. 
The stabilizer ${\rm Stab}(d)[\Phi]$ 
of $d$ in ${\rm Mod}(\Sigma_3)[\Phi]$ 
is the quotient of the product of
two subgroups $G_1,G_2$ by an infinite cyclic central subgroup.
The group $G_1$ is 
the group of all 
isotopy classes of diffeomorphisms of $\Sigma_3$ which
fix the bordered surface $S=\Sigma_3-T$ pointwise and preserve the
spin structure $\Phi$.
It is isomorphic to 
the subgroup of the mapping class group of a one-holed torus
which preserves the spin structure $\Phi$.
The group $G_2$ is the group of all isotopy classes of
diffeomorphisms of $\Sigma_3$ which fix $T$ pointwise and preserve
the spin structure $\Phi$. The center of ${\rm Stab}(d)[\Phi]$
is generated by a Dehn twist $T_d$ about $d$.

Consider the curve system ${\cal A}_4\subset {\cal E}_6$ which consists of the
curves $c_0,c_3,c_4,c_5$.
It is contained in the subsurface $\Sigma_2^1=\Sigma_3-T$ of $\Sigma_3$ 
of genus 2 which
is bounded by $d$. 
The Dehn twists about these curves generate
a subgroup ${\cal A}({\cal A}_4)$
of $\Gamma\cap G_2$ which is isomorphic to the braid group in
five strands  
(see \cite{FM12} or \cite{Ma00} for the last statement). 
By Theorem 1.4 of \cite{Ma00}, the Dehn twist $T_d$ can be
represented as an explicit word in the Dehn twists about the curves
from the curve system ${\cal A}(A_4)$. In particular, we have 
$T_d\in \Gamma$.

Let $\Sigma_{2,1}$ be the surface obtained from
$\Sigma_2^1=\Sigma_3-T$ by replacing 
the boundary component by a puncture, and let $\Sigma_2$ be
obtained from $\Sigma_{2,1}$ by forgetting the puncture. 
Let $\phi$ be the $\mathbb{Z}/2\mathbb{Z}$-reduction of the spin structure $\Phi$.
The spin structure $\phi$ induces an odd spin structure on 
$\Sigma_{2,1}$ and $\Sigma_2$, again denoted by $\phi$.
By Lemma \ref{genus2}, the subgroup ${\cal A}({\cal A}_4)$ of $\Gamma\cap G_1$ 
surjects onto the spin mapping class group ${\rm Mod}(\Sigma_2)$.
Consequently 
the restriction of the puncture forgetful
homomorphism $G_2\to {\rm Mod}(\Sigma_2)[\phi]$ to $\Gamma\cap G_2$
is surjective. 

By homological coherence, if we orient $d$ as the oriented boundary of 
the surface $\Sigma_3-T$, then we have $\Phi(d)=1$. Thus by 
Lemma \ref{pointpush1}, the intersection of the point pushing group
$\pi_1(\Sigma_2)$ with the stabilizer of $\Phi$ in ${\rm Mod}(\Sigma_{2,1})$ 
is the preimage of the sublattice $\Lambda$ of $H_1(\Sigma_2,\mathbb{Z})$ 
generated by squares of primitive homology classes of oriented simple closed curves
under the natural homomorphism
$\pi_1(\Sigma_2)\to H_1(\Sigma_2,\mathbb{Z})$.
Or, equivalently, it equals the kernel of the surjective
homomorphism $\pi_1(\Sigma_2)\to H_1(\Sigma_2,\mathbb{Z}/2\mathbb{Z})$. 
In particular, ${\rm Mod}(\Sigma_{2,1})[\Phi]\cap \pi_1(\Sigma_2)$ contains 
the commutator
subgroup of $\pi_1(\Sigma_2)$.

We claim first that the square of the point pushing map along a 
simple closed curve $\alpha$
with $\Phi(\alpha)=\pm 1$ is contained in $\Gamma$. To this end note that
as $\Phi(\alpha)=\pm 1$ if and only if we have $\phi(\alpha)=1$ where
$\phi$ is the $\mathbb{Z}/2\mathbb{Z}$-reduction of $\Phi$,  
the group ${\rm Mod}(\Sigma_2)[\phi]$ and hence $\Gamma$ 
acts transitively on these curves. Thus by equivariance,  
it suffices to verify this claim for a single such curve.

Consider again the simple closed curve $a_5\subset \Sigma_{2,1}$
with $\Phi(a_5)=\pm 1$
which intersects $c_5$ in a single point and is disjoint from all other curves from the 
curve system ${\cal E}_6$. 
Let $a^\prime$ be the simple closed curve  which bounds with 
$a_5$ and the boundary circle
$C$ of $\Sigma_{2,1}$ 
a pair of pants, that is,
$a_5$ and $a^\prime$ bound a holed annulus in $\Sigma_2^1$. 
By the chain relation in the mapping class group
(see \cite{FM12}), 
we have
\[(T_{c_0}T_{c_3}T_{c_4})^6=T_{a_5}T_{a^\prime}=\zeta\in \Gamma.\]
On the other hand, Lemma \ref{dehntwist4} shows that 
$T_{a_5}^4\in \Gamma$. As $T_{a_5}$ and $T_{a^\prime}$ commute, we have
$T_{a_5}^{-4}\zeta^2=T_{a_5}^{-2}T_{a^\prime}^2\in \Gamma$,
and this is just the square of the point
pushing transformation (via replacing the boundary
circle $C$ by a puncture) along 
$a_5$. Thus the square of the point pushing transformation
about $a_5$ is contained in $\Gamma$, which is what we wanted to show.

Now the sublattice $\Lambda\subset H_1(\Sigma_2,\mathbb{Z})$ is additively 
generated by elements of the form $2b$ where $b$ is an oriented simple closed
curve with $\phi(b)=1$ and hence we conclude that $\Gamma\cap \pi_1(\Sigma_2)$
surjects onto $\Lambda$.

We are left with showing that the point pushing map along any element in the commutator 
subgroup of $\pi_1(\Sigma_2)$ is contained in $\Gamma$. 
The commutator 
subgroup of $\pi_1(\Sigma_2)$ is generated by separating simple closed
curves. Note that ${\rm Mod}(\Sigma_2)[\phi]$ acts transitively on separating simple closed curves 
in $\Sigma_2$. 
Namely, as the parity of $\phi$ is odd, the formula (\ref{arf}) for the Arf invariant shows that
any separating simple closed curve $c$ on $\Sigma_2$ decomposes $\Sigma_2$ into two one holed
tori $T_1,T_2$ such that up to exchanging $T_1$ and $T_2$, 
there is a geometric symplectic basis $\alpha_1,\beta_1$ for $T_1$ with 
$\phi(\alpha_1)=1,\phi(\beta_1)=0$, and a geometric symplectic basis 
$\alpha_2,\beta_2$ for $T_2$ with $\phi(\alpha_2)=\phi(\beta_2)=0$. 
Then transitivity of the action of ${\rm Mod}(\Sigma_2)[\phi]$ on separating simple closed
curves follows once again from \cite{HJ89}. Thus 
it suffices to show the following: there exists a separating
simple closed curve $e$ in $\Sigma_2$ such that the
point pushing map along $e$ 
in $\Sigma_2$ is contained 
in $\Gamma$.

Now by Theorem 1.4 of \cite{Ma00}, the Dehn twist about the separating simple closed
curve $d^\prime$ which intersects $c_4$ in two points and is disjoint from the remaining curves
from ${\cal E}_6$ is contained in $\Gamma$. This separating curve is odd in the
sense described above. The second separating curve which bounds together 
with the boundary circle $C$ and $d^\prime$ 
a pair of pants is the boundary of a tubular neighborhood of 
$c_0\cup c_1$. As the Dehn twists about $c_0,c_1$ are contained in $\Gamma$,
the same holds true for the Dehn twist about that curve. We conclude that 
the point pushing maps about separating simple closed curves is contained in 
$\Gamma$.

To summarize, the quotient of $\Gamma\cap G_2$ by the infinite
cyclic group of Dehn twists about the boundary curve $d$ 
contains a generating set
for the point pushing subgroup of $G_2/\mathbb{Z}$ and hence it
contains this 
point pushing subgroup. As $\Gamma\cap G_2$ surjects onto the quotient
$G_2/\mathbb{Z}$ by the point pushing subgroup, we conclude that
$\Gamma$ surjects onto $G_2/\mathbb{Z}$. But $\Gamma$ contains the
infinite cyclic center of $G_2$ and hence $\Gamma\cap G_2=G_2$.

To complete the proof of the lemma we are left with showing that the 
subgroup $G_1$ of ${\rm Mod}(\Sigma_3)[\Phi]$ is contained in $\Gamma$.
But conjugation by the involution in $\Gamma$ which acts as an involution on 
the curve diagram of the curve system ${\cal E}_6$ and exchanges $c_1$ and 
$c_5$ and $c_2$ and $c_4$ maps $G_1$ to a subgroup of $G_2$ and hence
to a subgroup of $\Gamma$. Thus $G_1\subset \Gamma$ as well and
we conclude that indeed,
$\Gamma\cap {\rm Stab}(d)={\rm Mod}(\Sigma_3)[\Phi]\cap {\rm Stab}(d)$.
\end{proof}

\begin{remark}
Theorem \ref{main3} classifies connected components of 
the preimage in the Teichm\"uller space of abelian differentials
of the odd component of the stratum of abelian differentials
on a surface $\Sigma_3$ of genus $3$ with a single zero. Those components
correspond precisely to odd $\mathbb{Z}/4\mathbb{Z}$-spin structures
on $\Sigma_3$.
\end{remark}

\begin{remark}
The results in this article give a general recipe for finding generators
of spin mapping class groups. This recipe is motivated by
the recent work on compactifications of strata of abelian 
differentials in \cite{BCGGM18} and the goal to obtain a topological
interpretation of this compactification. 
\end{remark}

\begin{appendix}
\section  
{Additional graphs of nonseparating curves with
  fixed spin value}\label{addition}

In this appendix we complement the
main result in Section \ref{graphsofcurves}
by studying connectedness of 
some additional geometrically defined graphs related to spin
structures. The proofs do not use new ideas. We use the
assumptions and notations from Section \ref{graphsofcurves}.

We begin with adding more constraints
to the graph ${\cal C\cal G}_1^+$.
Define a graph ${\cal C\cal G}_1^{++}$ as follows. Vertices of 
${\cal C\cal G}_1^{++}$ are ordered pairs $(c,d)$ of nonseparating
simple closed curves $c,d$
such that $\phi(c)=\pm 1, \phi(d)=0$ and that $c,d$ intersect in a single
point. Then $c\cup d$ fills a one-holed torus $T(c,d)\subset S$. 
Two such pairs 
$(c,d),(c^\prime,d^\prime)$ 
are
connected by an edge if and only if the tori
$T(c,d)$ and $T(c^\prime,d^\prime)$ are 
disjoint. We use Corollary \ref{connected} to show

 \begin{lemma}\label{next+}
For $g\geq 4$ the graph ${\cal C\cal G}_1^{++}$ is connected. 
\end{lemma}   
\begin{proof}
  Let $(a,b),(c,d)$ be two vertices in the graph
  ${\cal C\cal G}_1^{++}$.
  Then $a,c$ are vertices in the graph ${\cal C\cal G}_1^+$.
  Connect $a=a_0$ to $c=a_k$ by an edge path $(a_i)$ 
in ${\cal C\cal G}_1^+$; this is possible by Corollary \ref{connected}. 
Our goal is to construct inductively a path
$(c_j,d_j)\subset {\cal C\cal G}_1^{++}$ connecting
$(a,b)$ to $(c,d)$ which 
passes through vertices $(c_{j_i},d_{j_i})$ with $c_{j_i}=a_i$. 

To this end observe that if the curve 
$b$ is disjoint from $a_1$, then we can find a curve 
$\hat d_1$ which intersects $a_1$
in a single point and is disjoint from $(a,b)$.
In particular, $a\cup b$ is disjoint from $c_1\cup \hat d_1$.

We can not expect in general that $\phi(\hat d_1)=0$. 
However, as before, there exists some $k\in \mathbb{Z}$ such that
$\phi(T_{a_1}^k(\hat d_1))=0$. Define $c_1=a_1$ and 
$d_1=T^k_{c_1}(\hat d_1)$ and
note that $d_1$ is disjoint from $a\cup b$ and intersects
$c_1$ in a single point. Thus the pair $(c_1,d_1)$ is a vertex
in ${\cal C\cal G}_1^{++}$ which is connected to
$(a,b)$ by an edge.

Let us now assume 
that $b$ is not disjoint from $a_1$. Since
$b$ intersects $a$ in a single point, it determines a vertex 
in the nonseparating arc graph ${\cal A}(A_1,A_2)$ 
of $S-a$; here $A_1,A_2$ are the two boundary components of
$S-a$ which glue back to $a$. Denote this arc by $b_0$.

Connect $b_0$ to an arc $b^\prime$ disjoint from $a_1$
by an edge path $(b_i)$ in ${\cal A}(A_1,A_2)$. 
This is possible by Lemma \ref{connectedarc}. 
Cut $S-a$ open along $b_0$. The result is a
surface of genus $g-1\geq 3$ with connected boundary, 
and $S-(b\cup b_1)$ is a surface of genus $g-2\geq 2$ 
with two boundary components.

A surface of genus at least 2 
contains a nonseparating curve $u$ with $\phi(u)=1$,
and in fact it contains a pair
$(u,v)\in {\cal C\cal G}_1^{++}$. In other words, there exists a vertex
of ${\cal C\cal G}_1^{++}$ 
which is disjoint from $a,b,b_1$. Connect $(a,b)$ to $(a,b_1)$ by 
the edge path $(a,b)\to (u,v)\to (a,b_1)$ and proceed by induction. 
\end{proof}

Define a graph ${\cal D}$ as follows. Vertices are ordered pairs $(x,y)$ 
where $x$ is a vertex in ${\cal C\cal G}_1^{++}$ and where $y$ is a disjoint 
simple closed nonseparating cuve with $\phi(y)=0$.
Two such pairs are connected
by an edge if they can be realized disjointly. The following
observation is a straighforward application of 
Lemma \ref{next+} and the tools used so far. Its proof will be omitted.

\begin{lemma}\label{next++}
For $g\geq 4$ the graph ${\cal D}$ is connected. 
\end{lemma}

Define now a graph ${\cal C\cal G}_2^+$ as follows. Vertices are
pairs $(x,y)$ where $x$ is a nonseparating simple closed curve on $S$ with 
$\phi(x)=2$ and where $y$ is a simple closed curve
with $\phi(y)=0$ 
intersecting $x$ in a single point.
Two such vertices are connected by an edge of length one if and only if they
can be realized disjointly. 

We use the above constructions to show
 
\begin{proposition}\label{connected2}
For $g\geq 4$ the graph ${\cal C\cal G}_2^+$ is connected.
  \end{proposition}
  \begin{proof}
    Given a pair of disjoint simple closed curves $(c,d)$ with
    $\phi(c)=\pm 1$ and $\phi(d)=0$, cut $S$ open along $c,d$ and
    denote the boundary components of the resulting surface by
    $C_1,C_2,D_1,D_2$. For one of the two choices of $C_1,C_2$, say
    for $C_1$, the curve $c+_\epsilon d$ defined by any embedded
    arc $\epsilon$ connecting $C_1$ to either of $D_1,D_2$ satisfies
    $\phi(c+_\epsilon d)=\pm 2$.

    As a consequence,
    to any vertex $(c,d)\in {\cal D}$ we can associate in a
    non-deterministic way a vertex in ${\cal C\cal G}_2^+$ by
    replacing the simple closed curve $a$ with $\phi(a)=\pm 1$ in
    the pair which defines a vertex of ${\cal C\cal G}_1^{++}$
    to the simple closed curve component of the pair which defines
    a vertex in ${\cal D}$.
    
    Adjacent vertices may not give rise to disjoint curves,
   but this issue can be resolved  
      using a path in the nonseparating arc graph. Using the fact
      that
      the surface obtained by removing from $S$ a torus and cutting
      the resulting surface open along a nonseparating simple closed
      curve has genus at least $2$,
     we find for any
    two such arcs a disjoint curve $e$ with $\phi(e)=\pm 1$.
Connect $b$ to this curve with a disjoint arc. 
  \end{proof}

\end{appendix}

\bigskip

\noindent
MATHEMATISCHES INSTITUT DER UNIVERSIT\"AT BONN\\
Endenicher Allee 60,\\ 
D-53115 BONN, GERMANY

\bigskip\noindent
e-mail: ursula@math.uni-bonn.de
\end{document}